\documentclass[final]{amsart}

\usepackage[utf8]{inputenc} 

\usepackage{graphicx} 

\usepackage[foot]{amsaddr} 
\usepackage{verbatim} 
\usepackage{amssymb}
\usepackage{amsthm}
\usepackage{amsmath}
\usepackage{amsfonts}
\usepackage{latexsym}
\usepackage{mathtools} 
\usepackage{graphics}
\usepackage{float} 
\usepackage{enumitem} 
\usepackage{cite} 
\usepackage[english]{babel} 
\usepackage{bbm} 

\usepackage{esint}

\usepackage{calrsfs}

    \usepackage[backref=none]{hyperref} 
     \usepackage[usenames,dvipsnames]{color}
     \definecolor{MyDarkBlue}{rgb}{0,0.1,0.7}
     \hypersetup{pdfborder={0 0 0},colorlinks,breaklinks=true,
       urlcolor={black},citecolor={black},linkcolor={black}
     }

\usepackage[color]{showkeys}

\usepackage{color}

\theoremstyle{plain}
\newtheorem{theorem}{Theorem}

\newtheorem{definition}[theorem]{Definition}

\newtheorem{remark}[theorem]{Remark}

\numberwithin{equation}{section}
\numberwithin{theorem}{section}

\allowdisplaybreaks

\newcommand{\eqdef }{\overset{\mbox{\tiny{def}}}{=}}

\newcommand{\rone}{\mathbb{R}}

\title[Neural network solutions to the kinetic Fokker-Planck equation]{Trend to Equilibrium for the Kinetic Fokker-Planck Equation via the Neural Network Approach}

\author[H. J. Hwang]{Hyung Ju Hwang$^\dagger$}
\address{$^\dagger$Department of Mathematics, Pohang University of Science and Technology (POSTECH), Pohang 37673, Republic of Korea. 
}

\author[J. W. Jang]{Jin Woo Jang$^\ddagger$}
\address{$^\ddagger$Center for Geometry and Physics, Institute for Basic Science (IBS), Pohang 37673, Republic of Korea. 
}

\author[H. Jo]{Hyeontae Jo$^\dagger$}

\author[J. Y. Lee]{Jae Yong Lee$^\dagger$}

\AtBeginDocument{%
   \def\MR#1{}
}

\begin{document}



\let\thefootnote\relax\footnotetext{
	
	\textit{E-mail addresses:} hjhwang@postech.ac.kr (H. J. Hwang), jangjinw@ibs.re.kr (J. W. Jang), jht0116@postech.ac.kr (H. Jo), jaeyong@postech.ac.kr (J. Y. Lee)}
\addtocounter{footnote}{-1}\let\thefootnote\svthefootnote

\begin{abstract} The issue of the relaxation to equilibrium has been at the core of the kinetic theory of rarefied gas dynamics. In the paper, we introduce the Deep Neural Network (DNN) approximated solutions to the kinetic Fokker-Planck equation in a bounded interval and study the large-time asymptotic behavior of the solutions and other physically relevant macroscopic quantities. We impose the varied types of boundary conditions including the inflow-type and the reflection-type boundaries as well as the varied diffusion and friction coefficients and study the boundary effects on the asymptotic behaviors.  These include the predictions on the large-time behaviors of the pointwise values of the particle distribution and the macroscopic physical quantities including the total kinetic energy, the entropy, and the free energy. We also provide the theoretical supports for the pointwise convergence of the neural network solutions to the \textit{a priori} analytic solutions. We use the library \textit{PyTorch}, the activation function \textit{tanh} between layers, and the \textit{Adam} optimizer for the Deep Learning algorithm.
\end{abstract}

\setcounter{tocdepth}{2}

\maketitle
\tableofcontents

\thispagestyle{empty}

\section{Introduction}\subsection{Motivation}
One of the main questions of interest in the study of the dynamics of rarefied gas particles is on the time-asymptotic behaviors of the particle density distribution and its macroscopic quantities. Since the era of Ludwig Boltzmann, the validity of the time-irreversibility and the entropy production of the Boltzmann kinetic equation has long been a bone of contention due to the Poincar\'e recurrence theorem.
Indeed, it is an important problem to show that the time-scale of the convergence towards the equilibrium is much smaller than the time-scale of the validity of the Boltzmann equation. 
 
In this work, we study the time-irreversibility and the entropy production of the kinetic Fokker-Planck equation, which is a fundamental model for a physical plasma. This question has been heavily studied in both analytic and numerical aspects. The Lyapunov functional for the kinetic Fokker-Planck equation is given by the relative entropy functional with respect to the steady-state, which we will describe in more detail below, and we provide a newly-devised numerical method of using a machine learning algorithm for the study of the large-data asymptotic behaviors of the Deep Neural Network (DNN) solutions to the kinetic Fokker-Planck equation in a bounded domain. 

In fact, it has been considered numerically difficult to simulate an initial boundary value problem for a kinetic partial differential equation. One of the difficulties arises from the fact that a numerical solution is closely related to a computation in a bounded domain by nature, whereas the boundary of a kinetic equation still consists of the whole space in the momentum variable $v$. In order to resolve this issue, people (c.f. \cite{MR1910805}) have considered the decaying properties of the solutions to a kinetic equation; i.e., if one could observe that the solution decays fastly enough in time outside a compact set, then the size (in a specific sense) of solution outside the compact region is close to zero so one can treat the difference as a small error numerically.  
Another difficulty arises from the huge computational cost on the simulation of a kinetic partial differential equation due to the extra dimensions from the momentum variable $v$. This becomes worse when we consider an integral-based operator such as the classical Boltzmann collision operator or the Landau-Boltzmann collision operator.

The Deep Learning method is a new approach for solving partial differential equations that can resolve some of the issues that we listed above. The Deep Learning algorithm has an advantage of being intuitive and easy to be executed via the backpropagation method. The algorithm produces approximated solutions that can be differentiated continuously on domains. Also, it is relatively easier to put information in the algorithm by adding a term to the loss function via the Deep Learning approach. For example, we can simply include a term regarding the conservation of the total mass of the system in the total Loss function of the algorithm, as the scheme is proposed to conserve the total mass under several boundary conditions. In addition, the Deep Learning algorithm can be extended to arbitrary domains so it is not necessary to worry about how to split a domain into triangles as in the numerical methods. Although our DNN algorithm uses the uniform grids of domains, using sampling random points from a domain can be applied to a special domain in higher dimension equations \cite{sirignano2018dgm}.

However, there are also some weaknesses of the approach that one should be careful of.  Firstly, there is no guarantee that the Deep Learning algorithm will converge and it is theoretically difficult to show the convergence of the Deep Learning Algorithm. Also, it is hard to evaluate the accuracy of the Deep Learning algorithm in contrast with the numerical methods. So far, lots of varied measures are suggested to express the performance of a DNN model. Due to the randomly initialized parameters in a Deep Learning algorithm, each learning could give slightly different solutions, while the numerical methods are deterministic.

\subsection{A brief history of the past results} 
\subsubsection{Mathematical results on the Fokker-Planck equation}The existence and the uniqueness of the solutions to the Fokker-Planck equation have been heavily studied.
Dita \cite{Dita_1985} constructs an analytic solution of the stationary 1D Fokker-Planck equation with an absorbing boundary. Then Protopopescu \cite{Protopopescu_1987} deals with the stationary 1D Fokker-Planck equation under some velocity-dependent external forces with some boundary conditions.
DiPerna-Lions \cite{diperna1988fokker} established stability results for sequences of solutions and global existence for the Cauchy problem of the Fokker-Planck equation with large data.
Desvillettes and Villani \cite{desvillettes2001trend} showed that a polynomial decay for the solutions with suitable initial conditions to a global equilibrium with the help of logarithmic Sobolev inequalities.
The irregular coefficients in Fokker–Planck, and transport equations were also studied in \cite{lorenz2007radon, bris2008existence}.
Later, Mischler \cite{mischler2010kinetic} provided the stability results of DiPerna-Lions renormalized solutions for Fokker-Planck equations with Maxwell boundary conditions.
\cite{sheng2013well} showed that the well-posedness of the steady Fokker–Planck solutions in smooth domains with absorbing boundary conditions.

Regarding the hypoellipticity of the equation, Hwang-Jang-Velazquez in \cite{hwang2014fokker} showed the hypoellipticity properties for the Fokker-Planck equation with absorbing boundary conditions, and this has been generalized in \cite{hwang2018fokker} by Hwang-Jang-Jung.
Also, Hwang-Phan \cite{hwang2017fokker} extended the results of \cite{hwang2014fokker} to the inflow boundary conditions.

 Regarding the Vlasov-Poisson-Fokker-Planck system, Victory and O'Dwyer \cite{victory1990classical} proved the existence of local in time solutions to the VPFP system.  Neunzert, Pulvirenti, and Triolo in \cite{neunzert1984vlasov} used a probabilistic method to prove the global existence of smooth solutions in one and two dimensions. Degond and Pierre \cite{degond1986global} proposed a fully deterministic proof of the existence of global in time smooth solutions for the Vlasov-Fokker-Planck equations in one and two dimensions. They also proved that the solution of the VPFP equation converges to the solution of the VP equation as the coefficients in the FP operator term go to zero. Bouchut in \cite{bouchut1993existence,bouchut1995smoothing} showed the existence and uniqueness of strong and global in time solutions to the three-dimensional VPFP equation. The asymptotic behavior and the convergence to the equilibrium of the solutions to the Vlasov(-Poisson)-Fokker-Planck equation were studied in \cite{MR1639292,MR1414375,MR1343393,MR2765745,MR1470927}. The stationary states and large time behavior of the Wigner-Fokker-Planck equation were studied in \cite{MR2974172}.
The low and high field scaling limits of Vlasov-Poisson-Fokker-Planck system were considered in \cite{MR1848592}. The global existence and uniqueness of weak solutions to kinetic
Kolmogorov–Vicsek models were considered in \cite{MR3519972}. The Vlasov-Poisson-Fokker-Planck system with uncertainty and multiple scales was studied in \cite{MR3780745,MR3715369}.
Regarding the recent development in the qualitative properties of the VPFP system, F. Bouchut and J. Dolbeault in\cite{bouchut1995long} showed the large time behaviors and the steady-states for the solutions of the VPFP equation in the case that the particles occupy the whole space $\rone^3$. Another related result is in \cite{MR1470927} which studied the large time asymptotics for the VPFP system in a bounded domain with the reflection type boundary conditions.

\subsubsection{Numerical results on kinetic equations}In this section, we would like to introduce a few past results on the numerical analysis of kinetic equations.  
Regarding past numerical results on the Fokker-Planck equation, we would like to start with some early developments via the conservative finite element method \cite{MR892257,MR1447091,MR1283340,MR1739113,MR816660}. Regarding the (spatially homogeneous and inhomogeneous) nonlinear Landau collision equation, which is a generalized version of the linear Fokker-Planck equation, we have the early developments via the conservative finite element method \cite{MR610857,MR1677589,MR1640174,MR1688993,MR1606249,MR2186367}. We also record a result via the spectral method \cite{MR1906573}.

We note that many methods have been proposed to solve the Vlasov-Poisson-Fokker-Planck equation and the Fokker-Planck-Landau equation. Allen-Victory \cite{allen1994computational} and Havlak-Victory \cite{havlak1996numerical} proposed the random particle method with the analysis and the computational study of its method. The finite difference scheme was also used to solve the VPFP with the periodic 1D case in \cite{cheng1976integration,schaeffer1998convergence}. Another approach is the deterministic particle methods which are based on the characteristic trajectories for the transport term of the Vlasov-Poisson-Fokker-Planck equation \cite{havlak1998deterministic,schaeffer1997difference,schaeffer1998convergence}. Wollman and Ozizmir \cite{wollman2005numerical,wollman2008deterministic} combined the deterministic method with the periodic regriding of the distribution function to approximate the solutions more stable and accurate. This method is extended to the two-dimensional case in \cite{wollman2009numerical}. The fast spectral method, which is also an alternative method for the numerical approximation was introduced in \cite{pareschi2000fast}. In \cite{filbet2002numerical}, Filbet and Pareschi presented a new spectral method that could be extended to the nonhomogeneous situation.

\subsubsection{Neural networks and the Cauchy problem of a PDE } 

The neural network architecture was first introduced in \cite{mcculloch1943logical}. Then, Cybenko \cite{cybenko1989approximation} established sufficient conditions for which a continuous function can be approximated by finite linear combinations of single hidden layer neural networks with the same univariate function. Hornik-Stinchcombe-White \cite{hornik1989multilayer} also showed measurable functions can be approximated by the multi-layer feedforward networks with a monotone sigmoid function. Then Cotter \cite{cotter1990stone} extended the result of \cite{hornik1989multilayer} to a new architecture, and later Li \cite{li1996simultaneous} proved that the multi-layer network with one hidden layer can approximate a target function and its higher partial derivatives on a compact set.

Though the theory of artificial neural networks as an approximation to solutions of differential equations has such a long history, the actual implementation of the ideas has a relatively short history due to the technical and algorithmical issues. Solving differential equations using an artificial neural network with architecture including one single layer and ten units were studied in \cite{lagaris1998artificial}, and the results were extended to a domain with complex boundaries in \cite{lagaris2000neural}. Then Jianyu et al. \cite{jianyu2003numerical} replaced an activation function with a radial basis and solved the Poisson equation.

More recently, Berg-Nystr\"om \cite{berg2018unified} used a DNN to solve steady problems in 1D and 2D space dimensions with complex geometry. Han-Jentzen-Weinan \cite{han2018solving} then applied DNNs to the stochastic process for solving high dimensional differential equations. Very recently, Raissi-Perdikaris-Karniadakis \cite{raissi2019physics} suggested an algorithm that solves both forward and inverse problems. Sequentially, Jo et al. \cite{jo2019deep} gave a theoretical reason that neural networks converge to analytic solutions in forward and inverse problems. Also, the application of other architectures or learning strategies of a convolutional neural network or reinforcement learning has been studied in \cite{siahkoohi2019neural, wei2019general}.
Also, the neural network approaches to solve a partial differential equation have been proposed in \cite{berg2018unified,raissi2017physics,sirignano2018dgm}.

\subsection{The Fokker-Planck equation}The $d-$dimensional kinetic Fokker-Planck equation reads as
\begin{equation}
	\partial_t f +v\cdot \nabla_x f= \nabla_v\cdot (\sigma \nabla_v f +\beta vf), \ (t,x,v)\in [0,T]\times \Omega\times \mathbb{R}^d,
\end{equation}where $\Omega\subset \mathbb{R}^d$, $\sigma>0$ is the diffusion coefficient, $\beta\ge 0$ is the friction coefficient, and $f=f(t,x,v)$ is the probablistic density distribution of particles. In this paper, we consider the following 1-dimensional kinetic Fokker-Planck equation in a bounded interval $\Omega=(-1,1)$ and $d=1$ as
\begin{equation}\label{FPeq}
	\partial_t f +v \partial_x f= \partial_v(\sigma \partial_v f +\beta vf), \ (t,x,v)\in (0,T)\times \Omega \times \mathbb{R},
\end{equation} subject to the initial condition
\begin{equation}\label{initial}
f(0,x,v)=f_0(x,v)\geq 0, \ (x,v)\in \Omega\times  \mathbb{R}.
\end{equation}
The boundary conditions that we impose will be introduced in Section \ref{sec:boundary}.

 \subsection{Boundary conditions}\label{sec:boundary}We define the outward normal vector $n_x$ on the boundary $\partial \Omega$ as 
 \begin{equation}\label{out}
 n_x\eqdef  \frac{\nabla \zeta(x)}{|\nabla\zeta(x)|}.
 \end{equation}
 Since our $\Omega=(-1,1)$ and hence $\partial\Omega = \{-1,1\}$, our $n_x=\hat{x}$ at $x=1$ and $=-\hat{x}$ at $x=-1$.
 Throughout this paper, we will denote the phase boundary of $\partial\Omega\times \rone$ as $\gamma\eqdef  \partial\Omega\times \rone.$ Additionally we split this boundary into an outgoing boundary $\gamma_+$, an incoming boundary $\gamma_-$, and a singular boundary $\gamma_0$ for grazing velocities, defined as
 \begin{equation}
 \begin{split}
 \gamma_+&\eqdef  \{(x,v)\in\Omega\times \rone :n_x\cdot v >0\},\\
 \gamma_-&\eqdef  \{(x,v)\in\Omega\times \rone :n_x\cdot v <0\},\\
 \gamma_0&\eqdef  \{(x,v)\in\Omega\times \rone :n_x\cdot v =0\}.
 \end{split}
 \end{equation}
 Since our $\Omega =(-1,1)$ and $\partial\Omega = \{-1,1\}$, we have 
  \begin{equation}
 \begin{split}
 \gamma_+&=\{(1,v)|\ v>0\}\cup\{(-1,v)|\ v<0\} ,\\
 \gamma_-&=\{(1,v)|\ v<0\}\cup\{(-1,v)|\ v>0\},\\
 \gamma_0&=\{(1,0), (-1,0)\} .
 \end{split}
 \end{equation}

 In terms of the probability density function $f$, we formulate the following four physical boundary conditions throughout the paper. 
 \subsubsection{Specular reflection boundary condition}  \begin{equation}\label{specular}
 f(t,x,v)|_{\gamma_-}=f(t,x,-v),
 \end{equation} for $x=-1$ and $1$.
 
 \subsubsection{Diffusive reflection boundary condition}  \begin{equation}\label{diffusive}
 f(t,x,v)|_{\gamma_-}=C\mu(v)\int_{w\cdot n_x>0}f(t,x,w)|w\cdot n_x|dw,
 \end{equation} for $x=-1$ and $1$, where $$C=\left(\int_{v\cdot n<0}\mu(v)|v\cdot n|dv\right)^{-1}\  \text{and} \ \mu(v)\eqdef e^{-\frac{v^2}{2}},$$
for both $n=\hat{x}$ and $=-\hat{x}$. 
 \subsubsection{Periodic boundary condition}  \begin{equation}\label{periodic}
 f(t,x,v)|_{\gamma_-}=f(t,-x,v),
 \end{equation} for $x=-1$ and $1$.
 
 \subsubsection{Absorbing boundary condition}  \begin{equation}\label{absorbing}
 f(t,x,v)|_{\gamma_-}=0,
 \end{equation} for $x=-1$ and $1$.
 
 \subsubsection{Inflow boundary condition}  \begin{equation}\label{inflow}
 f(t,x,v)|_{\gamma_-}=g(t,x,v),\ \text{for} \ x=-1\text{ and }1,
 \end{equation} with a given function $g$.
 
\subsection{The equilibrium state, the Lyapunov functional, and the balance laws}\label{equilibrium}It is well-known that the linear Fokker-Planck equation \eqref{FPeq} has a global equilibrium solution. This is called the global Maxwellian solution and the form of the steady-state was introduced in \cite[Theorem 1.2, p. 1349]{MR1470927} as follows:
\begin{equation}\label{maxwellian}
f_\infty (v)=\frac{M}{C(2\pi\frac{\sigma}{\beta})^{0.5}}\exp\left(-\frac{\beta}{\sigma}\frac{|v|^2}{2}\right)
\end{equation}
where $M=\|f_0(\cdot,\cdot)\|_{L^1_{x,v}}$ and $C=|\Omega|=2$.
The Lyapunov functional $\eta(t)$ is defined by the relative entropy of the solution $f$ with respect to the stationary distribution $f_\infty$ as 
\begin{equation}
\begin{split}
\eta(t)&\eqdef \int_{\Omega\times \rone} f\log\left(\frac{f}{f_\infty}\right)dxdv\\
&=\int_{\Omega\times \rone} f\log f dxdv +\frac{\beta}{2\sigma} \int_{\Omega\times \rone} |v|^2f dxdv +\log\left(\frac{C(2\pi\frac{\sigma}{\beta})^{0.5}}{M}\right) \int_{\Omega\times \rone} f dxdv\\
&=-\text{Ent}(t) +\frac{\beta}{\sigma}\text{KE} (t)+\log\left(\frac{C(2\pi\frac{\sigma}{\beta})^{0.5}}{M}\right)\text{Mass}(t) ,\end{split}\end{equation} where the entropy of the system ``\text{Ent}", the total kinetic energy ``\text{KE}", and the total mass ``Mass" of the system are defined as
$$\text{Ent}(t)\eqdef- \int_{\Omega\times \rone} f\log f dxdv,$$
$$\text{KE}(t) \eqdef \frac{1}{2} \int_{\Omega\times \rone} |v|^2f dxdv,$$ and $$
\text{Mass}(t)\eqdef \int_{\Omega\times \rone} f dxdv.$$ We define the free energy functional ``FE" as $$\text{FE}(t) \eqdef \text{KE}(t)-\frac{\sigma}{\beta} \text{Ent}(t),$$ whose time-derivative is equivalent to that of the Lyapunov functional $\eta$ if the total mass is conserved. In the cases of the specular reflection, the diffusive reflection, and the periodic boundaries, we expect that the Lyapunov functional satisfies $\eta'(t)\le 0$, which is a manifestation of the second law of thermodynamics.

The following balance laws on the macroscopic quantities \cite[p. 1350]{MR1470927} are also well-known:
\begin{itemize}
  \item Balance of the total mass:
  \begin{equation}
      \frac{d}{dt}\|f(t,\cdot,\cdot )\|_{L^1_{x,v}} = -\int_{(-1,1)\times \rone} dxdv (v \cdot n_x)\gamma f(t,x,v).
  \end{equation}
  \item Balance of the total kinetic energy:
  \begin{multline}\label{Balance_KE}
      \frac{d}{dt} \left( \frac{1}{2} \int_{(-1,1)\times \rone} dxdv |v|^2 f(t,x,v) \right) \\ = -\frac{1}{2} \int_{\{-1,1\}\times \rone} dxdv (v \cdot n_x) |v|^2 f(t,x,v) \\- \beta\int_{(-1,1)\times \rone} dxdv |v|^2 f(t,x,v) + \sigma \|f(0,\cdot,\cdot )\|_{L^1_{x,v}}.
  \end{multline}
  \item Balance of the entropy:
  \begin{multline}
      \frac{d}{dt} \int_{(-1,1)\times \rone} dxdv (-f(t,x,v) \log f)(t,x,v) \\ =- \beta\|f(0,\cdot,\cdot )\|_{L^1_{x,v}} + \int_{\{-1,1\}\times \rone} dxdv (v \cdot n_x) f(t,x,v) \log f(t,x,v) \\+ 4\sigma\int_{(-1,1)\times \rone} dxdv \left| \nabla_v \sqrt{f} \right|^2.
  \end{multline}
\end{itemize}
\subsection{Main results, difficulties, and our strategy}The main results of the paper consist of two parts. One is on the theoretical evidence on the relationship between the DNN solutions and the \textit{a priori} analytic solutions. More precisely, we first prove in Theorem \ref{forward_loss} that a sequence of neural network approximated solutions, which make the total loss term converge to zero, exists if a $\widehat{C}^{(1,1,2)}$ solution to the equation exists. Namely, the theorem implies we can always find appropriate weights of the neural nework which reduce the total error functions as much as we want. However, since this does not guarantee the convergence of the approximated solutions as the total loss term converges to zero, we prove an additional Theorem \ref{theorem_forward} which provides us that the neural network solutions indeed converge to the analytic solution as the total loss term vanishes. This is proved under the suitable smallness on the boundary condition \eqref{artificial_bdry} which corresponds to the truncation of the $v$ domain in the choice of our grid points, as introduced in Section \ref{assumption_for_nn}. Though the equation is linear, the derivation of an energy inequality for the convergence includes the boundary integrations not just in $x$ variable but also in $v$ variable due to the truncation of $v$ space and hence we needed the smallness condition on the difference of the approximated solutions and the analytic solution as in the condition \eqref{artificial_bdry} to deal with the boundary integrations in $v$ variable. Also, in order to create a sufficiently large dissipation term on the left-hand side of the energy inequality, we follow the transformation of the Fokker-Planck equation motivated by Carrillo \cite{carrillo1998global}. 

On the other hand, we provide in Section \ref{sec:results} the time-asymptotic behaviors of the neural network approximated solutions and the macroscopic physical quantities to the kinetic Fokker-Planck equation under the various types of initial and boundary conditions with different values of diffusion and friction coefficients. The learning algorithms of our DNN algorithm are based on the development of the total loss function that contains the information of the Fokker-Planck equation and the initial-boundary conditions and the appropriate selection of the grid points and the weights for each loss term so they fit to the physical conditions.
Our algorithm uses the library \textit{PyTorch} and the hyper-tangent \textit{tanh} activation function between the layers. For the purpose of minimizing the total loss term, we use the \textit{Adam} optimizer which is based on the stochastic gradient descent. One of the reasons that many of the existing numerical methods fail to simulate the solutions to several kinetic equations is on the issue of the conservation of the total mass. In order to deal with the difficulty, we remark that we increased the learning efficiency of the model by adding the mass conservation law to the total loss function. It is indeed a huge advantage of using the neural network approach that we can intuitively add any physically relevant conditions on the necessary quantities when we design the total loss function, though the reduction of the total loss function is another issue of concern that we need to take care of.

 In addition, our numerical simulations indeed provide several predictions on the time-asymptotic behaviors of the \textit{a priori} analytic solutions to the kinetic Fokker-Planck equation and the physically relevant macroscopic quantities in the pointwise sense. Namely, we provide the time-asymptotic plots on the actual pointwise value at each grid under varied types of conditions, and these definitely contain much more information than the graphs on the asymptotic behaviors of the general weighted $L^p$-moments of the solution for some $p$. To the best of authors' knowledge, there have been no numerical plots of the pointwise values of the approximated solutions for a kinetic equation under the varied types of the physical boundary conditions, such as the inflow-type and the reflection-type boundary conditions. 
Via the numerical simulations, we have observed the asymptotic behaviors of the solutions to the equation that were theoretically proved under some specific situations; our numerical simulations predict the pointwise convergence of the neural network solutions to the global Maxwellian for varied types of the boundary conditions and the varied coefficients for the diffusion and the friction. Under the specular reflection boundary condition, we also provide the different rates of convergence under the varied choices of the friction coefficient $\beta$.

\subsection{Outline of the paper}In Section \ref{sec:neural}, we will introduce in detail our DNN architecture and method to solve the Cauchy problem to the Fokker-Planck equation. This will include the detailed descriptions on the hidden layers and the definitions of grid points and loss functions. In Section \ref{sec:convergence}, we will first prove that there exists a sequence of weights such that the total sum of loss functions converges to 0 and that neural networks equipped with such weights converge to the analytic solutions. In Section \ref{sec:results}, we provide our numerical simulations and the result for each initial and boundary condition. Several plots will show the actual values of the distribution function at each time and spatial variable and will provide the asymptotic behaviors of macroscopic quantities as well as the actual values of the distribution. Finally, in Section \ref{sec:conclusion}, we complete the paper by summarizing our methods and the results.
\section{Methodology: Neural Network approach}\label{sec:neural}
In this section, we introduce our DNN structure and method to solve the Cauchy problem to the 1-dimensional kinetic Fokker-Planck equation \eqref{FPeq}.

\subsection{Our Deep Learning algorithm and the architecture}
A Deep Learning algorithm is a non-linear function approximation method via a DNN structure. A DNN consists of several layers, and each layer has several neurons which are connected to pre- and post- layer neurons. Connected neurons have a relation with an affine transformation and a non-linear activation function. We denote the approximated function as $f^{nn}(t,x,v;m,w,b)$ and we suppose our DNN has $L$ layers; in other words, our DNN has an input layer, $L-1$ hidden layers, and an output layer. The input layer takes $(t,x,v)$ as input and the final layer gives $f^{nn}(t,x,v;m,w,b)$ as the output. We denote the relation between the $l$-th layer and the $(l+1)$-th layer ($l=1,2,...,L-1$) as
\begin{equation*}
	z_j^{(l+1)}=\sum_{i=1}^{m_l}w_{ji}^{(l+1)}\bar{\sigma}_{l}(z_i^l) + b_j^{(l+1)},
\end{equation*}
where $m=(m_0,m_1,m_2,...,m_{L-1})$, $w=\{w^{(k)}_{ji}\}_{i,j,k=1}^{m_{k-1},m_k,L},$ $b=\{b^{(k)}_j\}_{j=1,k=1}^{m_k,L},$ and
\begin{itemize}
	\item $z_i^l$ : the $i$-th neuron in the $l$-th layer
	\item $\bar{\sigma}_l$ : the activation function in the $l$-th layer
	\item $w_{ji}^{(l+1)}$ : the weight between the $i$-th neuron in the $l$-th layer and the $j$-th neuron in the $(l+1)$-th layer
	\item $b_j^{(l+1)}$ : the bias of the $j$-th neuron in the $(l+1)$-th layer
	\item $m_l$ : the number of neurons in the $l$-th layer.
\end{itemize}
Note that the relation between the input layer and the first-hidden layer is expressed as follows:
\begin{equation*}
	z_j^{1}=\sum_{i=1}^3 w_{ji}^1 z_i^0 + b_j^{1},
\end{equation*}
where $(z_1^0, z_2^0, z_3^0)=(t,x,v)$.

We use the library \textit{PyTorch} and the hyper-tangent \textit{tanh} activation function for our fully connected DNN. Our DNN has four layers which each layer has 3-128-256-128-2 neurons as shown in the following figure:
\begin{figure}[H]
	\includegraphics[width=0.5\textwidth, draft=false]{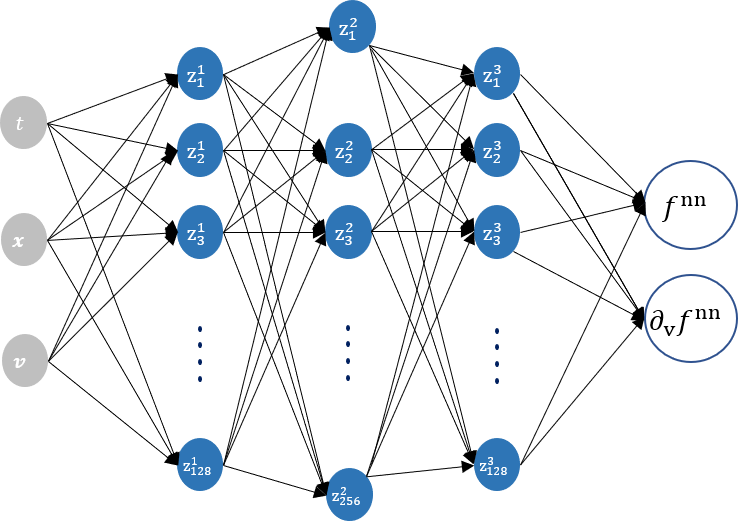}
	\caption{The DNN structure}
	\label{fig:DNN_structure}
\end{figure}
 The reason that we make the second output is to approximate $\partial_{vv} f^{nn}(t,x,v;m,w,b)$ via the reduction of order technique. We approximate $\partial_{vv} f^{nn}(t,x,v;m,w,b)$ by $$\partial_v h^{nn}(t,x,v;m,w,b)$$ using the second output $$\partial_v f^{nn}(t,x,v;m,w,b)=h^{nn}(t,x,v;m,w,b)$$ contained in the loss term; see Section \ref{Loss functions}. It reduces the computational costs dramatically. Regarding the optimization algorithm, we use \textit{Adam} optimization algorithm, which is an extended algorithm of the stochastic gradient descent and is heavily used in the applications of the deep learning. 

\subsection{Grid points}\label{assumption_for_nn}
To approximate the kinetic solution $f(t,x,v)$ via the Deep Learning algorithm, we make the data of grid points for each variable domain.  We choose $T$ as $5$ or $10$ for each boundary condition and each collision coefficient. We truncate the momentum space for the $v$ variable as $[-20,20]$, make the grid points for $v$ in $[-10,10]$ for the training, and assume that $f^{nn}(t,x,v;m,w,b)$ is $0$ if $|v|>10$. More precisely, the grid points for the training are chosen uniformly as follows:
$$
	\left\{ (t_i,x_j,v_k) \right\}_{i,j,k} \in [0,T] \times [-1,1] \times [-10,10]\  \text{with} \ \Delta t=0.01 ,\ \Delta x=0.02,\ \Delta v=0.2.
$$
We use the grids $$\left\{ (t=0,x_j,v_k) \right\}_{j,k}$$ for the initial condition and $$\left\{ (t_i,x=1\text{\ or\ }-1,v_k) \right\}_{i,k}$$ for the boundary condition.

\subsection{Loss functions}\label{Loss functions}
In the algorithm, the Adam optimizer finds the optimal parameters $w_{ji}^{(l+1)}$ and $b_j^{(l+1)}$ to minimize loss functions using the back-propagation method. Thus, we need to define loss functions for our 1-dimensional kinetic Fokker-Planck equation: $Loss_{GE}$ for the Fokker-Planck equation \eqref{FPeq}, $Loss_{IC}$ for the initial condition \eqref{initial} and $Loss_{BC}$ for the boundary conditions defined as in Section \ref{sec:boundary}.

Firstly, we define a loss function for the governing equation \eqref{FPeq}. We use the \textit{reduction-of-order} technique for the second-order term as follows:
\begin{align*}
  Loss_{GE}^1
  &= \int_{(0,T)}dt\int_{(-1,1)}dx\int_{V}dv |\partial_t f^{nn}(t,x,v;m,w,b) +v \partial_x f^{nn}(t,x,v;m,w,b)\\&\quad\quad\quad\quad - (\sigma\partial_v h^{nn}(t,x,v;m,w,b) + \beta \partial_v (vf^{nn})(t,x,v;m,w,b))|^2,\\
  Loss_{GE}^2
  &= \int_{(0,T)}dt\int_{(-1,1)}dx\int_{V}dv |h^{nn}(t,x,v;m,w,b)-\partial_v f^{nn}(t,x,v;m,w,b)|^2,
\end{align*}where $V\eqdef [-10,10].$
Then we define $Loss_{GE}$ as
\begin{multline}\label{loss_ge}
  Loss_{GE} = Loss_{GE}^1 + Loss_{GE}^2\\ \approx \frac{1}{N_{i,j,k}}\sum_{i,j,k} |\partial_t f^{nn}(t_i,x_j,v_k;m,w,b) +v \partial_x f^{nn}(t_i,x_j,v_k;m,w,b) \\ - (\sigma\partial_v h^{nn}(t_i,x_j,v_k;m,w,b) + \beta \partial_v (vf^{nn})(t_i,x_j,v_k;m,w,b))|^2 \\+ |h^{nn}(t_i,x_j,v_k;m,w,b)-\partial_v f^{nn}(t_i,x_j,v_k;m,w,b)|^2,
\end{multline}
where $N_{i,j,k}$ is the number of grid points.

We now define the loss function for the initial condition via the use of the initial grid points as
\begin{multline}\label{loss_ic}
  Loss_{IC} = \int_{(-1,1)}dx\int_{V} dv\left|f^{nn}(0,x,v)-f_0(x,v)\right|^2\\ \approx \frac{1}{N_{j,k}}\sum_{j,k} \left| f^{nn}(0,x_j,v_k)-f_0(x_j,v_k)\right|^2.
\end{multline}

The loss function for the \textit{inflow} boundary condition in Section \ref{sec:boundary} is defined as follows:
\begin{multline}\label{loss_bc}
  Loss_{BC} = \sum_{x\in\{-1,1\}}\int_{V} dv\left|f^{nn}(t,x,v;m,w,b)-g(t,x,v;m,w,b)\right|^2\\ \approx \frac{1}{2N_{i,k}}\sum_{x\in\{-1,1\},i,k} \left|f^{nn}(t_i,x,v_k;m,w,b)-g(t_i,x,v_k;m,w,b)\right|^2.
\end{multline} For the other types of the boundary conditions from Section \ref{sec:boundary}, we define the loss term similarly via altering $g(t,x,v;m,w,b)$ and $g(t_i,x,v_k;m,w,b)$. 

One of the well-known a priori conservation law for the Fokker-Planck equation (\ref{FPeq}) for the specular, the periodic and the diffusive boundary conditions is the conservation of mass, in which we have that the $L^1_{x,v}$ moment of $\|f^{nn}(t,\cdot,\cdot ;m,w,b)\|_{L^1_{x,v}}$ propagates in time.  Therefore, the reduction of the loss term \eqref{loss_ge} would result in the reduction of the loss term that is defined for the conservation of mass. Therefore, without loss of generality, we add one more loss term with repect to the conservation of mass for more accurate analysis when impose the three boundary conditions for each $t_i$:
\begin{equation*}
  Loss_{\text{Mass}}^i = \left| \frac{d}{dt}\int_{(-1,1)}\int_{V} f^{nn}(t,x,v;m,w,b) \right|^2_{t=t_i}.
\end{equation*}
Then we define the whole $Loss_{\text{Mass}}$ as
\begin{equation}\label{loss_mass}
  Loss_{\text{Mass}} = \sum_i Loss_{\text{Mass}}^i \approx \frac{1}{N_i} \sum_{i}  \left| \frac{1}{N_{j,k}}\frac{d}{dt}\sum_{j,k}  f^{nn}(t,x_j,v_k;m,w,b) \right|^2_{t=t_i}.
\end{equation}

Finally, we define the total loss as
\begin{equation}\label{loss_total}
  Loss_{Total} = Loss_{GE} + Loss_{IC} + Loss_{BC}
\end{equation}
for the absorbing boundary (\ref{absorbing}) and the inflow bounbary (\ref{inflow}), and
\begin{equation}\notag
  Loss_{Total} = Loss_{GE} + Loss_{IC} + Loss_{BC} + Loss_{\text{\text{Mass}}}
\end{equation}
for the rest of the boundary conditions of Section \ref{sec:boundary}.
\section{Theoretical discussions}
\subsection{On the convergence of DNN solutions to analytic solutions}\label{sec:convergence}
In this section, we prove that there exists a sequence of weights such that the total sum of loss functions, defined later as \eqref{loss_total}, converges to 0. Sequentially, we also prove that neural networks equipped with such weights converge to an analytic solution. 
Throughout the section, we assume that the existence and the uniqueness of solutions for \eqref{FPeq} and \eqref{initial} with either \eqref{inflow} or \eqref{specular} are a priori given. We first introduce the following definition and the theorem from \cite{li1996simultaneous} on the existence of the approximated neural network solution:
\begin{definition}[Li, \cite{li1996simultaneous}]\label{C_hat} For a compact set $K$ of $\mathbb{R}^n$, we say $f\in \widehat{C}^m(K)$, $m\in \mathbb{Z}_+^n$ if there is an open $\Omega$ (depending on $f$) such that $K\subset \Omega$ and $f\in C^m(\Omega).$
\end{definition}

\begin{theorem}[Li, Theorem 2.1, \cite{li1996simultaneous}]\label{global} Let $K$ be a compact subset of $\mathbb{R}^n$, $n\ge 1$, and $f\in\widehat{C}^{m_1}(K)\cap\widehat{C}^{m_2}(K)\cap \cdots \widehat{C}^{m_q}(K)$, where $m_i \in \mathbb{Z}^n_+$ for $1\le i\le q$. Also, let $\bar{\sigma}$ be any non-polynomial function in $C^l(\mathbb{R})$, where $l=\max\{|m_i|:1\le i\le q\}$. Then for any $\varepsilon>0,$ there is a network
$$f^{nn}(x)=\sum_{i=0}^\nu c_i\bar{\sigma}(\langle w_i,x\rangle +\theta_i), \ x\in \mathbb{R}^n,$$ where $c_i\in \mathbb{R},$ $w_i\in \mathbb{R}^n$, and $\theta_i\in \mathbb{R}$, $0\le i\le \nu$ such that 
$$\|D^kf-D^kf^{nn}\|_{L^{\infty}(K)}<\varepsilon,$$
for $k\in \mathbb{Z}^n_+$, $k\le m_i$, for some $i$, $1\le i\le q.$
\end{theorem}
\begin{remark} We can generalize the result above to the one with several hidden layers (see, \cite{hornik1989multilayer}). Also, we may assume that the architecture is assumed to have only one hidden layer; i.e., $L=2$. 
\end{remark}
Now we introduce our first main theorem which states that a sequence of neural network solutions that makes the total loss term converge to zero exists if a $\widehat{C}^{(1,1,2)}$ solution to the equation exists:
\begin{theorem}\label{forward_loss}Assume that the solution $f$ to \eqref{FPeq} and \eqref{initial} with either \eqref{inflow} or \eqref{specular} belongs to $\widehat{C}^{(1,1,2)}([0,T]\times[-1,1]\times V)$, and the activation function $\bar{\sigma}(x)\in C^{(2,2,3)}([0,T]\times [-1,1]\times V)$ is non-polynomial. Then, there exists $\{m_{[j]}, w_{[j]}, b_{[j]}\}_{j=1}^\infty$ such that a sequence of the DNN solutions with $m_{[j]}$ nodes, denoted by $\{f_j(t,x,v) = f^{nn}(t,x,v;m_{[j]}, w_{[j]}, b_{[j]})\}_{j=1}^{\infty}$ satisfies
\begin{equation}\label{forward_loss_0}
Loss_{Total}(f_j)\rightarrow 0\text{ as }j\rightarrow\infty.
\end{equation}

\end{theorem}

\begin{proof}
Let $\epsilon > 0$ be given. By Theorem \ref{global}, there exists a neural network $$f_{j}(t,x,v) = \sum_{i=1}^{m_{[j],1}} w_{[j],1i}^{(2)} \bar{\sigma}((w_{[j],i1}^{(1)},w_{[j],i2}^{(1)},w_{[j],i3}^{(1)})\cdot(t,x,v)+ b^{(1)}_{[j],i}) + b^{(2)}_{[j],i},$$ such that $\|D^{\underline{k}}f - D^{\underline{k}}f_{j}\|_{L^{\infty}([0,T]\times[-1,1]\times V)} < \epsilon$, where ${\underline{k}}\leq (1,1,2)$ is a multi-index. By integrating $$| \partial_{t}f_j+v\partial_{x}f_j-\partial_{v}(\sigma\partial_{v}+\beta v)f_j|^2$$  and $$ |f_j(0,\cdot,\cdot)- f_0|^{2}$$ over $[0,T]\times [-1,1]\times V\text{, } [-1,1]\times V$, respectively, we obtain that the loss terms \eqref{loss_ge} and \eqref{loss_ic} are bounded by $(1+\mu(V)+\sigma+\beta\mu(V))T^{2}\mu(V)^{2}\varepsilon$ and $2\mu(V)^{2}\varepsilon$, where $\mu(\cdot)$ is the Lebesgue measure on $\mathbb{R}$. Now we assume that the boundary condition satisfies \eqref{inflow}. Then we note that all the boundary values are bounded by a supremum of interior values, since $f$ and $f_j$ are sufficiently smooth. That is,
\begin{align}
\|f_j - g \|_{L^{2}(\gamma^{-}_{T,V})}^{2} = \|f_j - f \|_{L^{2}(\gamma^{-}_{T,V})}^{2} \leq 2T^{2}\mu(V)^{2}\|f_j - f \|_{L^{\infty}(\gamma^{-}_{T,V})}  \nonumber\\
\leq 2T^{2}\mu(V)^{2}\|f_j - f \|_{L^{\infty}([0,T]\times[-1,1]\times V)} \leq 2T^{2}\mu(V)^{2}\varepsilon ,\nonumber
\end{align}where $\gamma^{+}_{T,V}$ and $\gamma^{-}_{T,V}$ are defined as 
\begin{equation}\label{setgammapm}\gamma^{\pm}_{T,V}\eqdef [0,T]\times \gamma_{\pm,V},\end{equation} where
$\gamma_{\pm,V}$ denote the boundarie sets $\gamma_+$ and $\gamma_-$ with the $v$ domain truncated to $V$.
This completes the proof by setting $\varepsilon = \varepsilon_j = \frac{1}{j}$.
\end{proof}
\begin{remark}
The assumption $f\in\widehat{C}^{(1,1,2)}([0,T]\times[-1,1]\times V)$ can be replaced by a general Sobolev space, since the functions in a Sobolev space can be approximated by the continuous functions on a compact set.
\end{remark}
We also remark that Theorem \ref{forward_loss} provides us that we can find the weight of the neural network that reduces the error function as much as we want. However, this does not guarantee that the neural network could converge to the solution of the original equation when the loss function converges to zero, though we will also discuss how to find the weights in the forthcoming sections. Due to the reason, we introduce our second theorem, Theorem \ref{theorem_forward}, which shows that the neural network architecture converges to an analytic solution in a suitable function space when the weights of the neural networks minimize $Loss_{Total}$. We prove it under the following additional condition at the boundary that is consistent to the truncation of the domain in $v$ variable introduced in Section \ref{assumption_for_nn}:
\begin{equation}\label{artificial_bdry}
|\partial^{k}_{v}f(t,x,v)-\partial^{k}_{v}f_j(t,x,v)|\leq \varepsilon \text{ if } v\in \partial V,
\end{equation}for some sufficiently small $\varepsilon>0$ and $k=0, 1$.
\begin{theorem}\label{theorem_forward} Let $\{m_{[j]}, w_{[j]}, b_{[j]}\}_{j=1}^\infty$ be a sequence defined in Theorem \ref{forward_loss} and let $f_j$ and $f$ satisfy \eqref{artificial_bdry}. Then, $Loss_{Total}(f_{j})\rightarrow 0$ implies
\begin{equation}
\left\| f_{j}(\cdot,\cdot,\cdot;m_{[j]},w_{[j]},b_{[j]})  - f\right\|_{L^{\infty}([0,T];L^{2}([-1,1]\times V))}\leq C(\sigma,\beta)\varepsilon^{2},
\end{equation}
where $f$ is a solution to  \eqref{FPeq} and \eqref{initial} with either \eqref{inflow} or \eqref{specular}, and $C(\sigma, \beta)$ is a non-negative constant depending only on $\sigma$ and $\beta$.
\end{theorem}
\begin{proof}Motivated by \cite{carrillo1998global}, we define a transform $\bar{u}(t,x,v)$ of a function $u(t,x,v)$ as follows:
$$\bar{u}(t,x,v) = e^{-(\lambda +\beta)t}u(t,x,e^{-\beta t}v)$$
where $\lambda$ is any non-negative constant that could control the convergence rate. Then the transformed function $\bar{f}$ satisfies $$\partial_t\bar{f} +e^{-\beta t}(v\cdot\partial_x)\bar{f} -\sigma e^{2\beta t}\partial_{v}^{2}\bar{f} +\lambda\bar{f}=0.$$ We now consider the following set of equations on the difference between $\bar{f}$ and $\bar{f_j}$ for each fixed $j$ as 
\begin{align}
&\left[\partial_{t}+e^{-\beta t}(v\cdot\partial_{x})-\sigma e^{2\beta t}\partial_{v}^{2}+\lambda \right]\{\bar{f}-\bar{f_j}\} \label{diff_ge}\\&\quad\quad\quad\quad= d_{ge,j}(t,x,v), \text{ for } (t,x,v)\in[0,T]\times [-1,1]\times e^{\beta t}V,\notag\\
&\bar{f}(0,x,v)-\bar{f_j}(0,x,v) = d_{ic,j}(x,v), \text{ for } (x,v)\in [-1,1]\times e^{\beta t}V, \label{diff_ic}\\
&\bar{f}(t,x,v)-\bar{f_j}(t,x,v)  = d_{bc,j}(t,x,v), \text{ for } (t,x,v) \in\gamma^{-}_{T,e^{\beta t}V},\label{diff_bc}
\end{align}
where the interval $e^{\beta t}V$ is defined as $$ e^{\beta t}V \eqdef  [-10e^{\beta t}, 10e^{\beta t}],$$ and $\gamma^{+}_{T,e^{\beta t}V}$ and $\gamma^{-}_{T,e^{\beta t}V}$ are defined as in \eqref{setgammapm}. Here, we define $$d_{ge,j}(t,x,v)\eqdef -\left[\partial_{t}+e^{-\beta t}(v\cdot\partial_{x})-\sigma e^{2\beta t}\partial_{v}^{2}+\lambda \right]\bar{f_j},$$ $$d_{ic,j}(x,v)\eqdef \bar{f}_0(x,v)-\bar{f_j}(0,x,v),$$and  $$d_{bc,j}\eqdef \bar{g}(t,x,v)-\bar{f_j}(t,x,v),$$ for the inflow boundary condition \eqref{inflow} and $$d_{bc,j}=\bar{f_j}(t,x,-v)-\bar{f_j}(t,x,v),$$ for the specular boundary condition \eqref{specular} instead. Then we derive the inequality below by multiplying $(\bar{f}-\bar{f_j})$ onto \eqref{diff_ge} and integrating it over $[-1,1]\times e^{\beta t}V$ as
\begin{multline}\label{energy_eq}
\int_{-1}^{1}\int_{-10e^{\beta t}}^{10e^{\beta t}}\frac{\partial}{\partial t}(\bar{f}-\bar{f_j})^{2}(t,x,v)dvdx - 2\sigma e^{2\beta t}\langle\partial_{v}^{2}(\bar{f}-\bar{f_j}),(\bar{f}-\bar{f_j})\rangle\\+2\int_{\gamma_{+,e^{\beta t}V}}(\bar{f}-\bar{f}_j)^{2}d\gamma + 2\lambda \| \bar{f}-\bar{f_j}\|^{2}_{L^{2}([-1,1]\times e^{\beta t}V)} \\
    = 2\int_{\gamma_{-,e^{\beta t}V}}d^{2}_{bc,j}d\gamma+2\langle d_{ge,j},(\bar{f}-\bar{f_j})\rangle, 
\end{multline}
where $\langle\cdot,\cdot\rangle$ denotes the standard inner product on ${L^{2}([-1,1]\times e^{\beta t}V)}$. On the left-hand side of \eqref{energy_eq}, we note that $$\int_{\gamma_{+,e^{\beta t}V}}|\bar{f}-\bar{f_j}|^{2}d\gamma\geq 0,$$ 
\begin{multline}\nonumber
\frac{d}{dt}\left\|(\bar{f}-\bar{f_j})(t,\cdot,\cdot)\right\|^{2}_{L^{2}([-1,1]\times e^{\beta t}V)} = \int_{-1}^{1}\int_{-10e^{\beta t}}^{10e^{\beta t}}\frac{\partial}{\partial t}(\bar{f}-\bar{f_j})^{2}(t,x,v)dvdx \\ + \underbrace{10\beta e^{\beta t}\left(\left\|(\bar{f}-\bar{f_j})(t,\cdot,10e^{\beta t})\right\|^{2}_{L^{2}([-1,1])} + \left\|(\bar{f}-\bar{f_j})(t,\cdot,-10e^{\beta t})\right\|^{2}_{L^{2}([-1,1])}\right)}_{\eqdef B_{1}(t)},
\end{multline}
by the Leibniz rule and
\begin{multline}\nonumber
2\sigma e^{2\beta t}\langle\partial_{v}^{2}(\bar{f}-\bar{f_j}),(\bar{f}-\bar{f_j})\rangle = -2\sigma e^{2\beta t} \left\|\partial_{v}(\bar{f}-\bar{f_j}) \right\|^{2}_{L^{2}([-1,1]\times e^{\beta t}V)} \\  +2\sigma e^{2\beta t}\int_{-1}^{1}\partial_{v}(\bar{f}-\bar{f_j})(\bar{f}-\bar{f_j})(t,\cdot,10e^{\beta t})dx \\-2\sigma e^{2\beta t} \int_{-1}^{1}\partial_{v}(\bar{f}-\bar{f_j})(\bar{f}-\bar{f_j})(t,\cdot,-10e^{\beta t})dx\\
\eqdef -2\sigma e^{2\beta t} \left\|\partial_{v}(\bar{f}-\bar{f_j}) \right\|^{2}_{L^{2}([-1,1]\times e^{\beta t}V)}+ B_{2}(t).
\end{multline}
Therefore, we reduce \eqref{energy_eq} to
\begin{multline}\label{energy_ineq}
\frac{d}{dt}\overbrace{\|\bar{f}-\bar{f_j}\|^{2}_{L^{2}([-1,1]\times e^{\beta t}V)}}^{Y(t)\eqdef } + 2\lambda \left\|\bar{f}-\bar{f_j}\right\|^{2}_{L^{2}([-1,1]\times e^{\beta t}V)} \leq \left\|\bar{f}-\bar{f_j}\right\|^{2}_{L^{2}([-1,1]\times e^{\beta t}V)} \\ + \underbrace{2 \int_{\gamma_{-,e^{\beta t}V}}d^{2}_{bc,j}d\gamma+\|d_{ge,j}\|^{2}_{L^{2}([-1,1]\times e^{\beta t}V)}}_{\eqdef L(t)} + B_{1}(t) + B_{2}(t).
\end{multline}
Multiplying \eqref{energy_ineq} by $e^{(2\lambda-1)t}$ and integrating it over $[0,t]$ for $t<T$, we have
\begin{equation}\label{conclusion}
Y(t)\leq Y(0)e^{-(2\lambda-1)t}+e^{-(2\lambda-1)t}\int_{0}^{t}e^{(2\lambda-1)s}\left(L(s)+B_{1}(s)+B_{2}(s)\right)ds
\end{equation}
Finally, we recall that $Y(t)=e^{-(2\lambda+\beta)t}\|f-f_j\|^{2}_{L^{2}([-1,1]\times V)}$, $Y(0)=Loss_{IC}$, and
\begin{multline}\nonumber
\int_{0}^{t}e^{(2\lambda-1)s}L(s)ds = \\ \int_{0}^{t}e^{(2\lambda-1)s}\left[2e^{-(2\lambda+\beta)s}\int_{\gamma_{-},V}(g-f_j)^2d\gamma+e^{-(2\lambda+\beta)s}\left\|Lf_j\right\|^{2}_{L^{2}((-1,1)\times V)}\right]ds \\ = \int_{0}^{t}\left[2e^{-(\beta+1)s}\int_{\gamma_{-},V}(g-f_j)^2d\gamma+e^{-(\beta+1)s}\left\|Lf_j\right\|^{2}_{L^{2}((-1,1)\times V)}\right]ds \\ \leq 2 (Loss_{BC}+Loss_{GE}),
\end{multline}
where $Lf_j = \left[\partial_{t} +v\partial_{x}-\partial_{v}(\sigma\partial_{v}+\beta v) \right]f_j$ is the Fokker-Planck operator. Moreover, under the assumption on \eqref{artificial_bdry}, we have
$$\int_{0}^{t}e^{(2\lambda-1)s}B_{1}(s)ds\leq 40\beta\varepsilon^{2} \int_{0}^{t}e^{(2\lambda-1)s}e^{-(2\lambda+\beta)s}ds\leq C_{1}(\beta)\varepsilon^{2},$$
$$
\int_{0}^{t}e^{(2\lambda-1)s}B_{2}(s)ds\leq 4\sigma\varepsilon^{2} \int_{0}^{t}e^{(2\lambda-1)s}e^{-(2\lambda+\beta)s}ds \leq C_{2}(\sigma,\beta)\varepsilon^{2}.$$
 Therefore, \eqref{conclusion} and the inverse transform from $\bar{f}$ to $f$ imply that
$$\|f-f_j\|^{2}_{L^{2}([-1,1]\times V)} \leq 2e^{(2\beta+1)t}(Loss_{IC}+Loss_{BC}+Loss_{GE}+C(\sigma,\beta)\varepsilon^{2}),$$
where $C(\sigma,\beta)=C_{1}(\beta)+C_{2}(\sigma,\beta)$. Since $t\in [0,T]$ and $Loss_{Total}(f_{j})\rightarrow 0$, this completes the proof of Theorem \ref{theorem_forward}. 
\end{proof}

\subsection{On the convergence rate of the kinetic energy}\label{sec:conrate}In this section, we would like to record the theoretical background on the convergence rate of the kinetic energy of the system, which will be observed via the neural network simulations as well in Section \ref{sec:results}.  This will be done only for the specular boundary condition \eqref{specular} in this section.
We define the kinetic energy functional of the solution $f$ to the Fokker-Planck equation as follows:
$$\text{KE}(t)\eqdef\frac{1}{2} \int_{(-1,1)\times \rone} dxdv |v|^2 f(t,x,v).$$
Then we can rewrite the balance of kinetic energy \eqref{Balance_KE} as follows:
\begin{equation}\label{dt_KE}
	\frac{d}{dt} \text{KE}(t) = -\frac{1}{2} \int_{\{-1,1\}\times \rone} dxdv (v \cdot n_x) |v|^2 f(t,x,v) - 2\beta \text{KE}(t) + \sigma M,
\end{equation}
where $M=\|f_0(\cdot,\cdot)\|_{L^1_{x,v}}.$ Then the first term on the right-hand side of \eqref{dt_KE} is equal to
\begin{align}
	&-\frac{1}{2} \int_{\{-1,1\}\times \rone} dxdv (v \cdot n_x) |v|^2 f(t,x,v) \notag \\
	=& -\frac{1}{2} \left( \int_{-\infty}^\infty dv \ (-v) |v|^2 f(t,x=-1,v) + \int_{-\infty}^\infty dv \ v |v|^2 f(t,x=1,v) \right).\label{dt_KE_2}
\end{align}
Under the specular boundary conditon \eqref{specular}, both two terms on the right-hand side of \eqref{dt_KE_2} is $0$, since
\begin{align*}
	&\int_{-\infty}^\infty dv (-v) |v|^2 f(t,x=-1,v)\\
	=& \int_{-\infty}^0 dv (-v) |v|^2 f(t,x=-1,v) + \int_0^\infty dv (-v) |v|^2 f(t,x=-1,v) \\
	=& \int_0^\infty dv\ v |v|^2 f(t,x=-1,-v) + \int_0^\infty dv (-v) |v|^2 f(t,x=-1,v) \\
	=& \int_0^\infty dv\ v |v|^2 f(t,x=-1,v) - \int_0^\infty dv\ v |v|^2 f(t,x=-1,v) = 0,
\end{align*}
and $\int_{-\infty}^\infty dv \ v |v|^2 f(t,x=1,v)$ is also $0$ in a similar manner. Therefore, we can rewrite \eqref{dt_KE} as
\begin{equation}\label{dt_KE_3}
	\frac{d}{dt} \text{KE}(t) + 2\beta \text{KE}(t) = \sigma M.
\end{equation}Therefore, we obtain
\begin{equation*}
	\text{KE}(t) e^{2\beta t} = \frac{\sigma M}{2\beta}e^{2\beta t} + C
\end{equation*}
by multiplying \eqref{dt_KE_3} by $e^{2\beta t}$ and integrating over $[0,t]$, where $C$ is a constant. We can get the closed-form of the kinetic energy under the specular boundary condtions for the kinetic Fokker-Planck equation as
\begin{equation*}
	\text{KE}(t) = \frac{\sigma M}{2\beta} + Ce^{-2\beta t}.
\end{equation*}
Now, for the total kinetic energy of the appoximated solution $f^{nn}(t,x,v;m,w,b)$ which is defined as
$$\text{KE}^{nn}(t;m,w,b)\eqdef\frac{1}{2} \int_{(-1,1)\times V} dxdv |v|^2 f(t,x,v),$$
the kinetic energy $\text{KE}^{nn}(t;m,w,b)$ satisfies
\begin{equation}\label{KE_func}
	\text{KE}^{nn}(t;m,w,b) \leq \text{KE}(t) = \frac{\sigma M}{2\beta} + Ce^{-2\beta t},
\end{equation}
since the truncated domain of velocity variable $V=[-10,10]$ is contained in $\rone$. This shows that the kinetic energy of the neural network solution converges faster to the steady-state value as the value of the coefficient $\beta$ gets larger.

\section{Neural Network Simulations}\label{sec:results}
In this section, we introduce the results of our neural network simulations for the DNN solution $f^{nn}(t,x,v;m,w,b)$ in three different ways. We first simulate our neural network algorithms for the varied boundary conditions, a given initial condition and the fixed values of $\sigma$, $\beta$ coefficients. In the cases of the specular reflection, the diffusive reflection, and the periodic boundaries, we expect that the Lyapunov functional (also called as the free energy or the relative entropy) satisfies $\eta'\le 0$, which is a manifestation of the second law of thermodynamics. Then, we alter the initial conditions for the specular boundary condition and for the fixed values of $\sigma$, $\beta$ coefficients. Lastly, we also obtain the results via altering the coefficients $\sigma$ and $\beta$ for a given initial condition and the specular boundary condition.

We analyze our neural network solution $f^{nn}(t,x,v;m,w,b)$ via computing the pointwise values at each grid-point, observing the $L^\infty$ norm of the solution and via computing the physical quantities of the total mass, the kinetic energy, the entropy, and the free energy. We compute these quantities via the approximations of the integration by the Riemann sum on the grid points, similarly to the previous subsection. For example, $L^1$ norm of $f^{nn}(t,x,v;m,w,b)$ for each time $t$ can be approximated as
$$\int_{(-1,1)}\int_{\rone} |f(t,x,v)| dxdv \approx \frac{1}{N_{j,k}} \sum_{j,k} |f(t,x_j,v_k)|.$$
We also compare our results with the existing thoretical results (c.f. \cite{MR1414375}). Especially for the specular \eqref{specular} and the diffusive \eqref{diffusive} boundary conditions which conserve the total mass, Bonilla-Carrillo-Soler \cite{MR1414375} gives the form of the steady-state as given in \eqref{maxwellian}.  The expected limiting values of kinetic energy, entropy, and free energy based on the global equilibrium \cite{wollman2008deterministic} are defined as follows:
\begin{align}
  \text{KE}_\infty
  &= \frac{\sigma M}{2\beta}, \label{KE_inf}\\
 \text {Ent}_\infty
  &= -M\log\left(\frac{M}{C(2\pi\frac{\sigma}{\beta})^{0.5}}\right)+\frac{1}{2}M\\
  \text{FE}_{\infty} &= \text{KE}_\infty -\frac{\sigma}{\beta}\text{Ent}_\infty  , \label{Ent_inf}
\end{align}
where $M=\|f_0(\cdot,\cdot)\|_{L^1_{x,v}}$ and $C=|\Omega|=2$ for our case.
Also, we calculate the entropy as the integral of $-f\log (f+10^{-10})$ to prevent from the divergence when $f(t,x,v)=0$.

We predict the pointwise values of the distribution functions $f^{nn}$ for each spatial variable $x$ in $[-1,1]$ and the velocity variable $v$ in $[-10,10]$ corresponding to the varied values of time variable $t$ in $[0,T]$. Namely, we observe the long-time behavior of neural network solution $f^{nn}(t,x,v;m,w,b)$ and check if the solution converges to Maxwellian \eqref{maxwellian} for the reflection-type boundaries. Note that minimizing the total loss does not guarantee a positivity of $f^{nn}$. Therefore, we also truncate $f^{nn}$ if $f^{nn}$ is sufficiently small ($<$ 0.005).

\subsection{Results by varying the boundary conditions}In this section, we impose different types of the boundary conditions that we introduced in Section \ref{sec:boundary}. Throughout this section, we set the values of the coefficients $\sigma$ and $\beta$ to be 1, and we consider the following \textit{cake-shaped} initial condition for the varied boundary conditions:
\begin{equation}\label{initial_cake}
f(0,x,v)=f_0(x,v) = \begin{cases}
1, &\text{ if $(x,v)\in (-0.9,0.9)\times  (-2,2),$}\\
0, &\text{ otherwise}.
\end{cases}
\end{equation}

We would first like to introduce the behavior of the approximated solutions under the absorbing boundary condition in the following section.
\subsubsection{The absorbing boundary condition}
By imposing the absorbing boundary condition (\ref{absorbing}), we would like to understand the dynamics of particles whose momenta vanish when they collide against the boundary. 
Figure \ref{fig:absorbing_physical} shows the $L^\infty$ norm, total mass, total kinetic energy, and the entropy of our neural network solution $f^{nn}(t,x,v;m,w,b)$ under the absorbing boundary condition. It has been shown in \cite{hwang2014fokker} that the $L^\infty$ norm and the total mass of the solutions $f(t,x,v)$ to the kinetic Fokker-Planck equation decay exponentially as the time variable $t$ goes to infinity when the friction coefficient $\beta$ is zero. In our case with the diffusion coefficient $\sigma=1$ and the friction coefficient $\beta=1$, the total mass decays in time as in Figure \ref{fig:absorbing_physical}, while the $L^\infty$ norm decays after an initial mixing and climbing.
\begin{figure}[H]
	\includegraphics[width=\textwidth, draft=false]{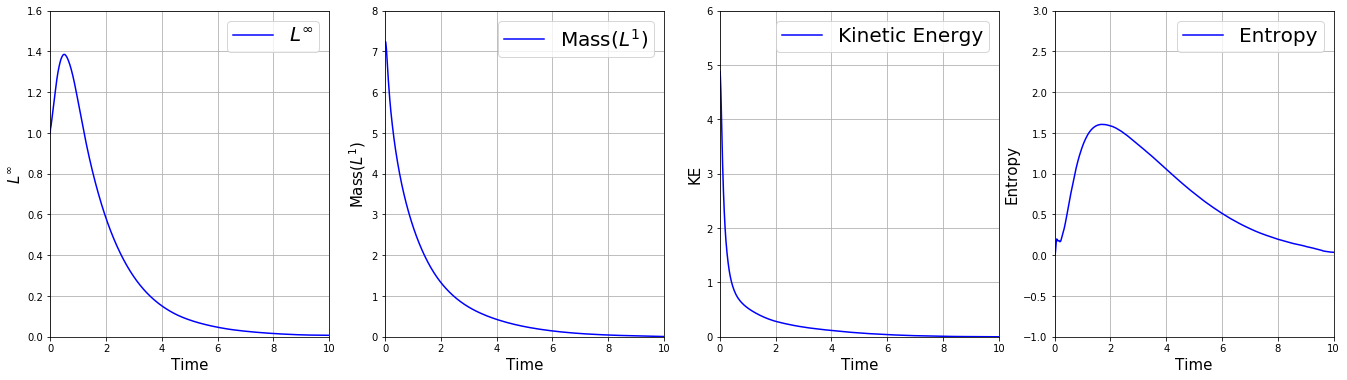}
	\caption{The time-asymptotic behavior of the $L^\infty$ norm and the macroscopic quantities of $f^{nn}(t,x,v;m,w,b)$ with the absorbing boundary condition. All quantities converges to zero as all the particles vanish once they reach the boundary.}
	\label{fig:absorbing_physical}
\end{figure}

Also, we can observe that our solution $f^{nn}(t,x,v;m,w,b)$ converges pointwisely to $0$, as shown in Figure \ref{fig:absorbing_f}. We remark that Figure \ref{fig:absorbing_f} shows each pointwise value of the neural network solution, and this definitely gives more information than just $L^1$ and $L^\infty$ values of the solution. 
\begin{figure}[H]
	\includegraphics[width=0.8\textwidth, draft=false]{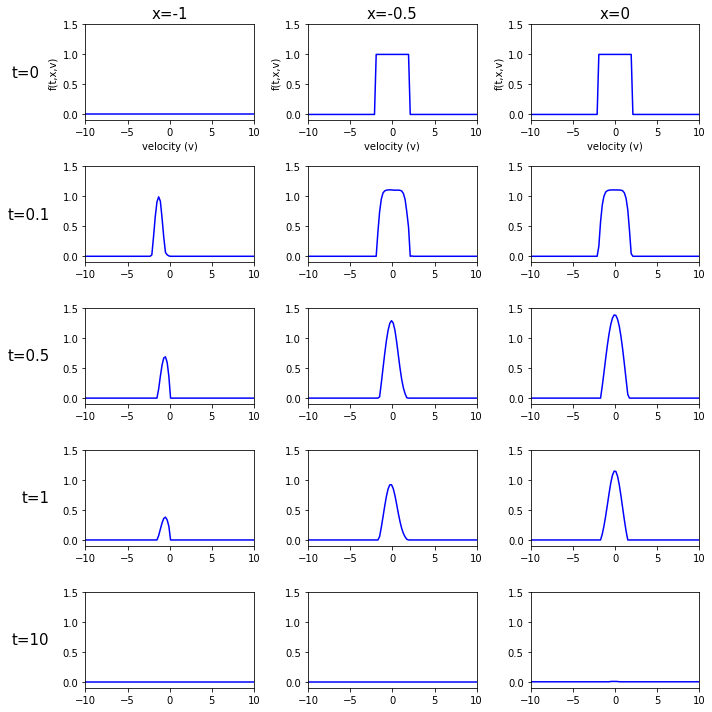}
	\caption{The pointwise values of $f^{nn}(t,x,v;m,w,b)$ as $t$ varies at each $x$'s for the absorbing boundary case. $x=-1$ stands for the boundary point, and $x=-0.5$ or $=0$ are the points away from the boundary. We omit the graphs for $x>0$ as we have obtained the symmetric graphs to $x<0$. }
	\label{fig:absorbing_f}
\end{figure}

\subsubsection{Inflow boundary condition}We now move onto the next boundary condition, the inflow boundary condition, with which we consider the situation that each boundary bounces particles in a given rate of the given function $g(t,x,v)$.  
We compare the pointwise values of the neural network solutions under three different inflow boundary conditions. The three different inflow boundary conditions are:
\begin{equation}\label{inflow_1}
 f(t,x,v)|_{\gamma_-}=g(t,x,v)=\frac{1}{2} \mathbbm{1}_{|v|\leq5},\ \text{for} \ x=-1\text{ and }1,
 \end{equation}
\begin{equation}\label{inflow_2}
 f(t,x,v)|_{\gamma_-}=g(t,x,v)=\begin{cases}\frac{1}{10} \mathbbm{1}_{|v|\leq5}, &\text{ if $x=-1$}\\
\frac{9}{10} \mathbbm{1}_{|v|\leq5}, &\text{ if $x=1$},
\end{cases}
 \end{equation}and
\begin{equation}\label{inflow_3}
 f(t,x,v)|_{\gamma_-}=g(t,x,v)=\frac{1}{2}e^{-t}\mathbbm{1}_{|v|\leq5},\ \text{for} \ x=-1\text{ and }1,
 \end{equation} where $\mathbbm{1}_{\mathbbm{S}}$ is the characteristic function on a set $\mathbbm{S}$.

In Figure \ref{fig:inflow_f}, the plot for the third type of inflow boundary condition (\ref{inflow_3}) has the similar pointwise values of $f^{nn}(t,x,v;m,w,b)$ to that of the first inflow boundary condition \eqref{inflow_1}, but it behaves like or converges to that of the absorbing boundary condition over time, as shown in the cyan-blue line. This converges pointwisely to $0$ after 10 time grids. On the other hand, we can observe that the other two types \eqref{inflow_1} and \eqref{inflow_2} converge to some distorted \textit{Maxwellian-like} shapes at each spatial position. Also, we can observe the smoothing effect of the kinetic Fokker-Planck equation in the interior domain at $x=-0.5,0$, and $0.5$. At the boundaries $x=-1$ and $x=1$, we observe the jump discontinuities of the momentum derivatives as expected due to the fixed inflow profiles. 
\begin{figure}[H]
	\includegraphics[width=\textwidth, draft=false]{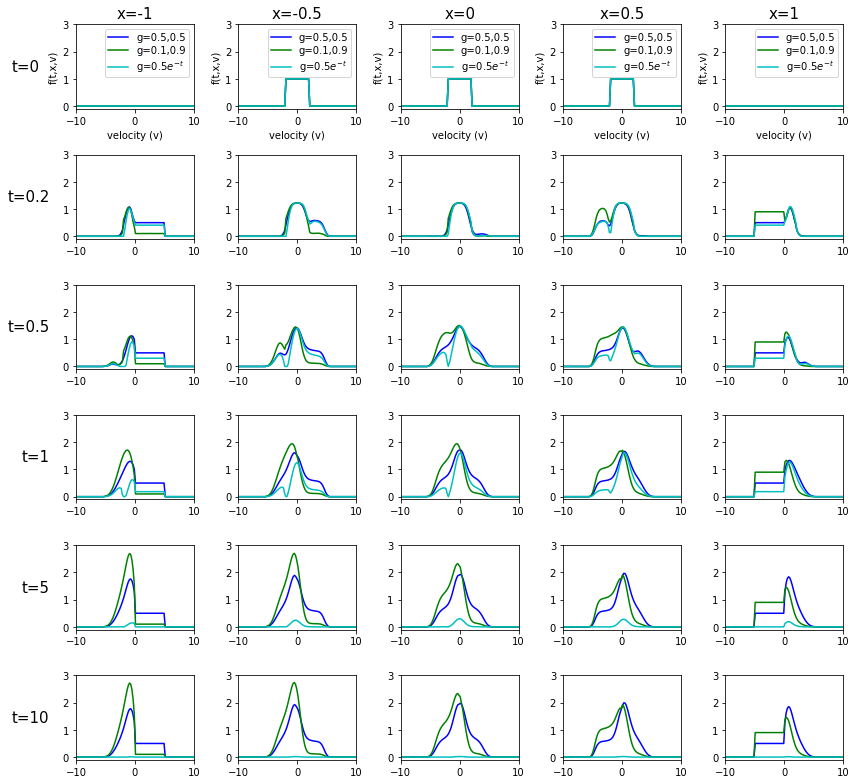}
	\caption{The pointwise values of $f^{nn}(t,x,v;m,w,b)$ as $t$ varies at each $x$'s for the different inflow boundary cases in different colors as shown in the legend. $x=\pm1$ stand for the boundary points, and $x=\pm 0.5$ and $=0$ are the points away from the boundary.}
	\label{fig:inflow_f}
\end{figure}

\subsubsection{The three boundary conditions which conserve the total mass}Now we would like to impose other types of the boundary conditions under which the total mass of the system is conserved. Namely, we will look at the dynamics of the particles under the specular boundary condition (\ref{specular}), the periodic boundary condition (\ref{periodic}), and the diffusive boundary condition (\ref{diffusive}).

We first remark that we obtain the similar behavior of the $L^\infty$ and $L^1$ norms of the solution in the three types of the boundary conditions. Though the convergence rates are slightly different, we could possibly say that the kinetic energies and the entropies of the three different cases converge to the same values in the red dotted line in Figure \ref{fig:conserve_physical}.
\begin{figure}[H]
	\includegraphics[width=\textwidth, draft=false]{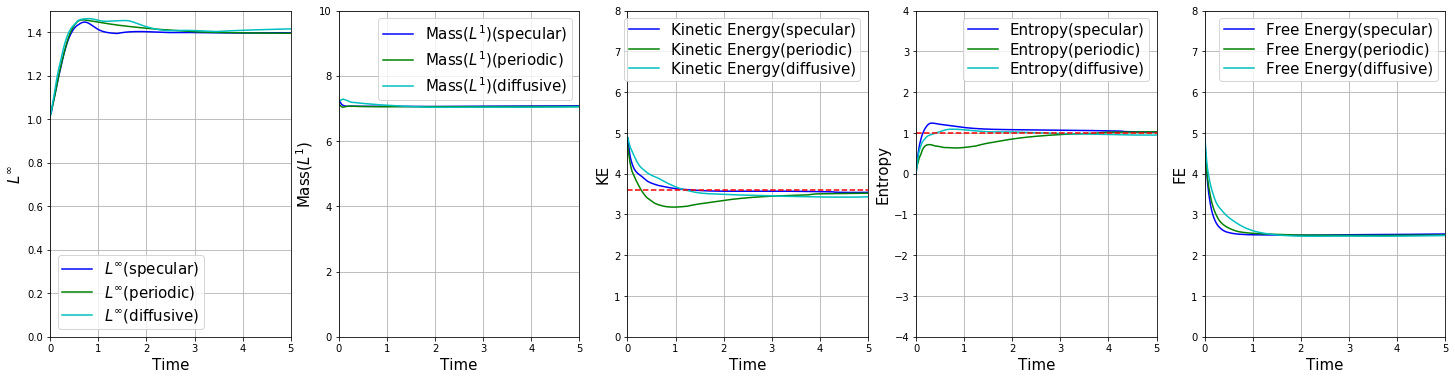}
	\caption{The time-asymptotic behaviors of the $L^\infty$ norms and the macroscopic quantities of $f^{nn}(t,x,v;m,w,b)$ with the specular, the periodic, and the diffusive boundary conditions, which conserve the total mass. The steady-state values of the kinetic energy (\ref{KE_inf}) and the entropy (\ref{Ent_inf}) are also given via the red-dotted lines. It is notable that the free energy (Lyapunov functional) is monotonically decreasing.}
	\label{fig:conserve_physical}
\end{figure}

In the perspective of the pointwise values of the neural network solutions, Figure \ref{fig:conserve_f} shows that the three cases converge pointwisely to the same steady-state solution; i.e., they converge to the global Maxwellian solution (\ref{maxwellian}). It is remarkable that the shape of the convergence to the global Maxwellian is slightly different in the diffusive boundary case from the other two conditions at the boundary point $x=-1.$ At the boundary $x=-1$, the periodic and the specular cases converge to the Maxwellian via the superposition of two waves from the left- and the right-hand sides at the same rate and hence via the creation of M-shaped plots. On the other hand, at the boundary $x=-1$, the solution in the case of the diffusive boundary condition converges to the global Maxwellian from the left to the right, as the diffusive boundary condition at $x=-1$ emits the Maxwellian-shaped values from the left to the right.
\begin{figure}
	\includegraphics[width=0.8\textwidth, draft=false]{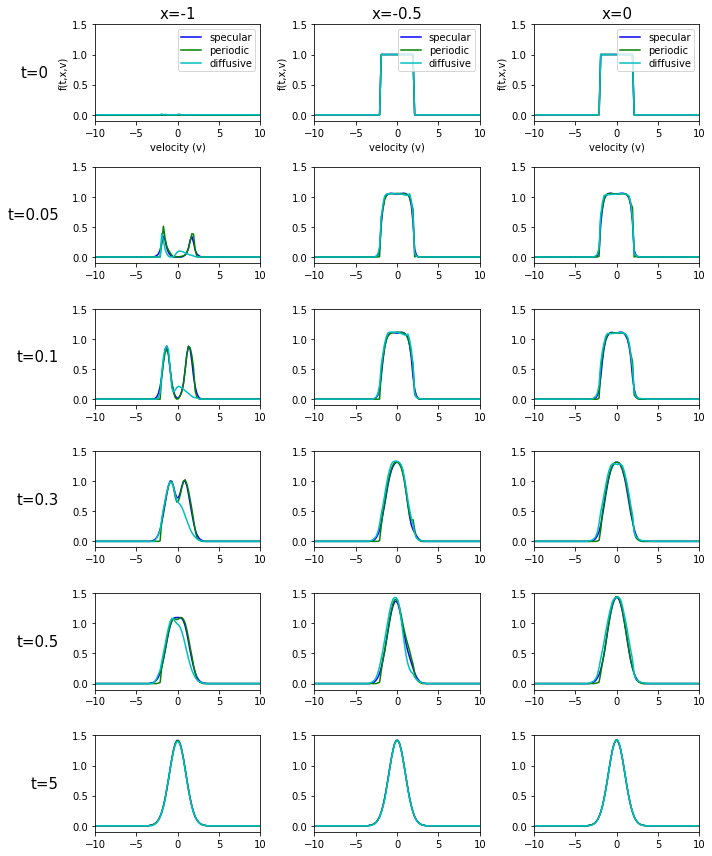}
	\caption{The pointwise values of $f^{nn}(t,x,v;m,w,b)$ as $t$ varies at each $x$'s for the specular, the periodic, and the diffusive boundary cases in different colors as shown in the legend. $x=-1$ stands for the boundary point, and $x=-0.5$ or $=0$ are the points away from the boundary. We omit the graphs for $x>0$ as we have obtained the symmetric graphs to $x<0$.}
	\label{fig:conserve_f}
\end{figure}

\subsection{Results by varying the initial conditions}
In this section, we vary the initial condition for the fixed coefficients $\sigma=\beta=1$ and the specular boundary condition. We impose the initial conditions as
\begin{equation}\label{initial_v2}
f(0,x,v)=f_0(x,v) = \begin{cases}
\frac{1}{25}v^2, &\text{ if $x\in (-0.9,0.9)$ and $v\in (-5,5)$}\\
0 ,&\text{ otherwise},
\end{cases}
\end{equation}and
\begin{equation}\label{initial_sin}
f(0,x,v)=f_0(x,v) = \begin{cases}
\sin(\frac{1}{v^2}), &\text{if $x\in (-0.9,0.9)$ and $v\in (-10,10)$}\\
0, &\text{otherwise}.
\end{cases}
\end{equation}In other words, we impose a M-shaped initial condition with the values compiled at large values in $v$ and a highly-oscillating initial condition near $v=0$. 

Figure \ref{fig:initial_f} shows the pointwise values of $f^{nn}(t,x,v;m,w,b)$ with the two initial conditions. We first note that the M-shaped distribution becomes the Maxwellian over time. Though the initial condition \eqref{initial_sin} is highly singular in its derivatives, Figure \ref{fig:initial_f} shows that $f^{nn}(t,x,v;m,w,b)$ is being regularized and converges to the Maxwellian over time. We remark that the two initial conditions result in two different Maxwellians and this difference comes from the different total masses of the two initial conditions.
\begin{figure}
	\includegraphics[width=0.8\textwidth, draft=false]{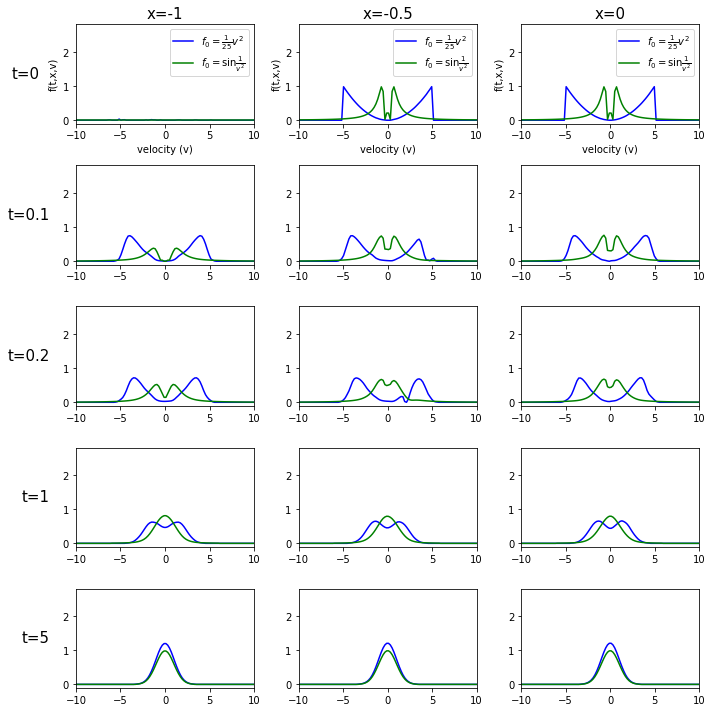}
	\caption{The pointwise values of $f^{nn}(t,x,v;m,w,b)$ as $t$ varies at each $x$'s for the specular boundary case with the special conditions drawn in different colors as shown in the legend. $x=-1$ stands for the boundary point, and $x=-0.5$ or $=0$ are the points away from the boundary. We omit the graphs for $x>0$ as we have obtained the symmetric graphs to $x<0$.}
	\label{fig:initial_f}
\end{figure}

\subsection{Results by varying the diffusion and the friction coefficients}
In this setion, we alter the two coefficients $\sigma$ and $\beta$ of the Fokker-Planck equation; the coefficients $\sigma$ and $\beta$ are related to the diffusion and the friction rates, respectively. Throughout the section, we fix the initial condition \eqref{initial_cake} and the specular boundary condition \eqref{specular}. 

We mention that, throughtout the section, we plot the values of the kinetic energy (\ref{KE_inf}) and the entropy (\ref{Ent_inf}) of the global Maxwellian solution for each value of $\sigma$ and $\beta$ in the red dotted lines. In the following three sub-sections, we alter the values of $\beta$ with a fixed $\sigma$, alter the values of $\sigma$ with a fixed $\beta$, and alter both $\sigma$ and $\beta$ by keeping the ratio $\frac{\sigma}{\beta}$ being the same.

\subsubsection{Different values of $\beta$}Throughout the section, we set $\sigma=1$ and let $\beta=2,1,0.5,$ and $0.25$. Figure \ref{fig:diff_beta_physical} shows that the kinetic energy and the entropy converge to the values of the red dotted lines of the expected steady-states for each different value of $\beta$. For example, the value of the kinetic energy for the steady-state with the coefficients $\sigma=1$ and $\beta=0.5$ is ${KE}_\infty = \frac{\sigma M}{2\beta} = 7.2$ by (\ref{KE_inf}), and the cyan-blue colored graph in the third plot of Figure \ref{fig:diff_beta_physical} converges to the red dotted line of the value 7.2. In addition, the graph for the $L^\infty$ norms shows that the $L^\infty$ values after a sufficiently large time become larger as $\beta$ gets larger and this agrees the expected $L^\infty$ values of the steady-state by \eqref{maxwellian}.
 
We remark that the rate of the convergence of the kinetic energy and the entropy for each different $\beta$ depends on the values of $\beta$; the solution $f^{nn}(t,x,v;m,w,b)$ with a larger value of the friction coefficient $\beta$ converged rapidly. This is also consistent to our theoretical supports provided in Section \ref{sec:conrate}.
\begin{figure}
	\includegraphics[width=\textwidth, draft=false]{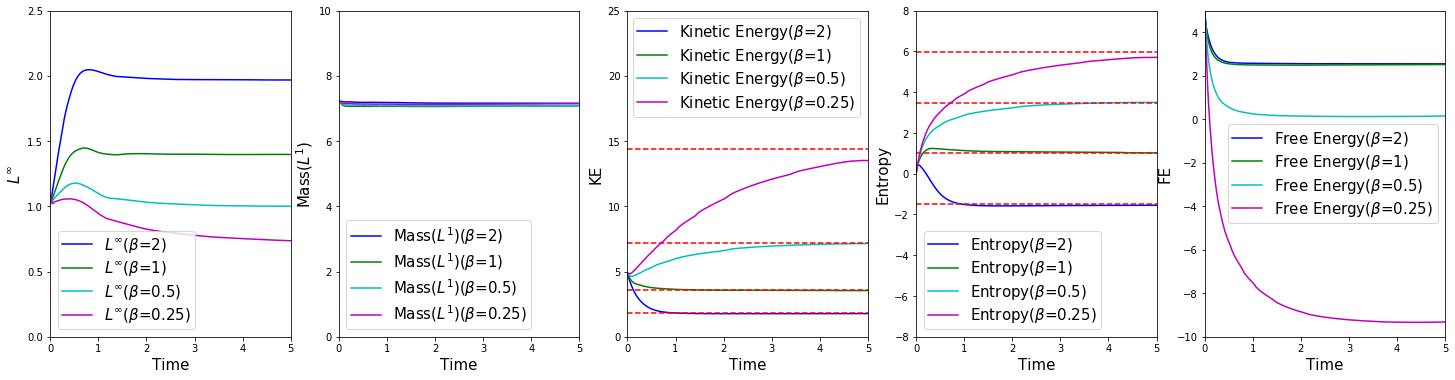}
	\caption{The time-asymptotic behaviors of the $L^\infty$ norms and the macroscopic quantities of $f^{nn}(t,x,v;m,w,b)$ for the specular boundary case as $\beta$ varies with the fixed $\sigma=1$. The steady-state values of the kinetic energy (\ref{KE_inf}) and the entropy (\ref{Ent_inf}) are also given via the red-dotted lines for each different $\beta$. It is notable that the free energy (Lyapunov functional) is monotonically decreasing.}
	\label{fig:diff_beta_physical}
\end{figure} 

Figure \ref{fig:diff_beta_f} shows that the pointwise values of $f^{nn}(t,x,v;m,w,b)$ approach the different steady-state solutions. Theoretically, the steady-state solution with the smaller friction coefficient $\beta$ makes the larger variance in the Maxwellian \eqref{maxwellian} with the same mean $0$, and Figure \ref{fig:diff_beta_f} agrees with the theory.
\begin{figure}
	\includegraphics[width=0.8\textwidth, draft=false]{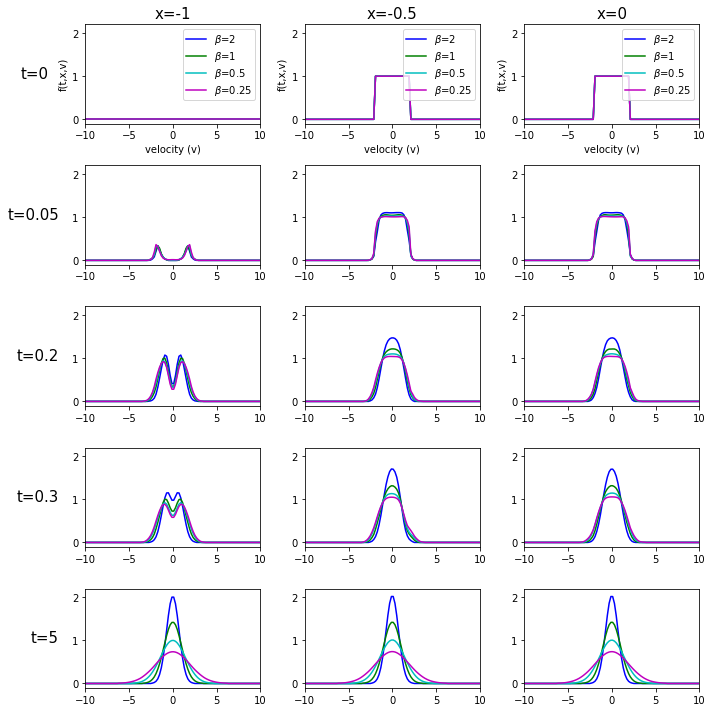}
	\caption{The pointwise values of $f^{nn}(t,x,v;m,w,b)$ as $t$ varies at each $x$'s for the specular boundary case, where $\sigma$ is fixed to be $=1$ and $\beta$ varies as shown in different colors as in the legend. $x=-1$ stands for the boundary point, and $x=-0.5$ or $=0$ are the points away from the boundary. We omit the graphs for $x>0$ as we have obtained the symmetric graphs to $x<0$.}
	\label{fig:diff_beta_f}
\end{figure}

\subsubsection{Different values of $\sigma$}In this section, we show the result with the fixed $\beta=1$ and the varied $\sigma$ coefficients of $\sigma=2,1,0.5,0.25$. Figure \ref{fig:diff_sigma_physical} shows that the physical quantities converge to those of the steady-state.
\begin{figure}
	\includegraphics[width=0.8\textwidth, draft=false]{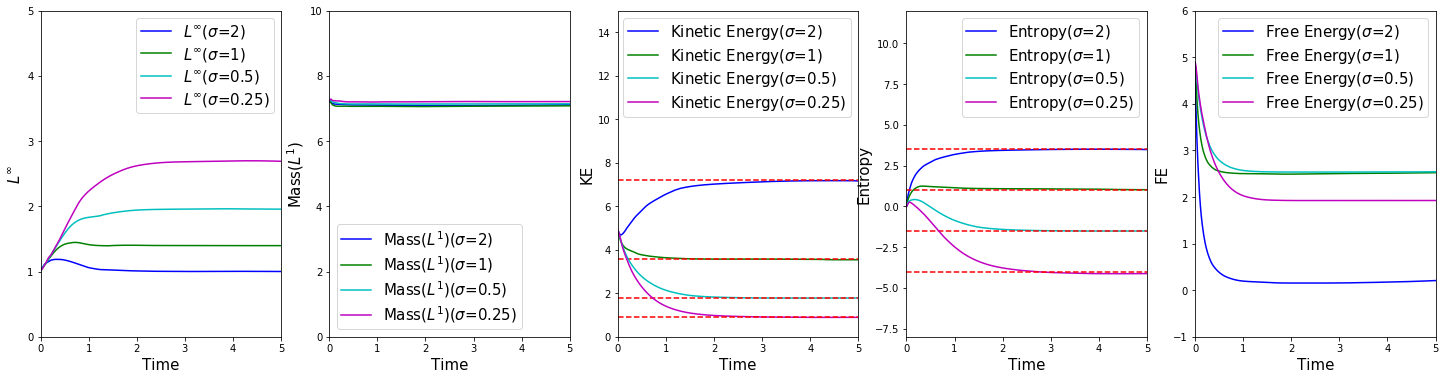}
	\caption{The time-asymptotic behaviors of the $L^\infty$ norms and the macroscopic quantities of $f^{nn}(t,x,v;m,w,b)$ for the specular boundary case as $\sigma$ varies with the fixed $\beta=1$. The steady-state values of the kinetic energy (\ref{KE_inf}) and the entropy (\ref{Ent_inf}) are also given via the red-dotted lines for each different $\sigma$. It is notable that the free energy (Lyapunov functional) is monotonically decreasing.}
	\label{fig:diff_sigma_physical}
\end{figure}Theoretically, the steady-state solution with the smaller diffusion coefficient $\sigma$ makes the smaller variance in the Maxwellian \ref{maxwellian} with the same mean $0$, and Figure \ref{fig:diff_sigma_f} agrees with the theory.
\begin{figure}
	\includegraphics[width=0.8\textwidth, draft=false]{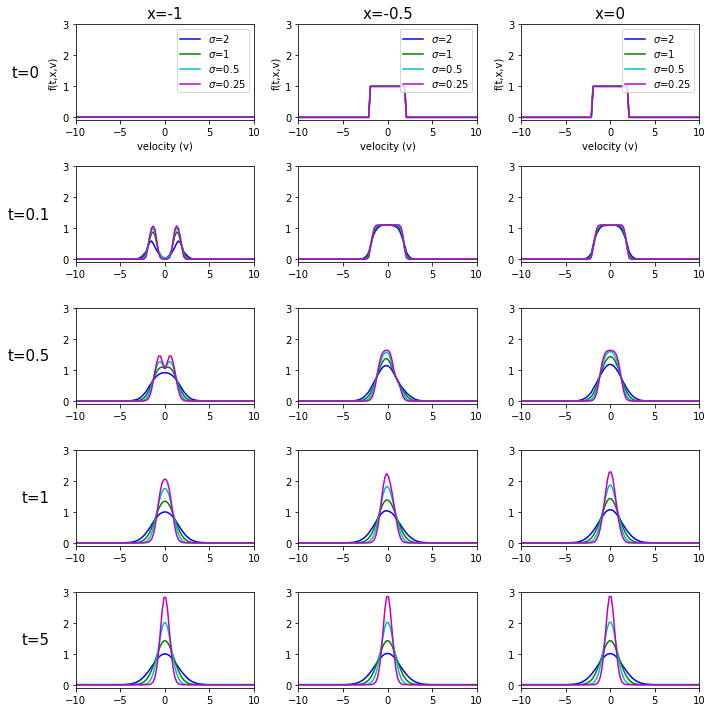}
	\caption{The pointwise values of $f^{nn}(t,x,v;m,w,b)$ as $t$ varies at each $x$'s for the specular boundary case, where $\beta$ is fixed to be $=1$ and $\sigma$ varies as shown in different colors as in the legend. $x=-1$ stands for the boundary point, and $x=-0.5$ or $=0$ are the points away from the boundary. We omit the graphs for $x>0$ as we have obtained the symmetric graphs to $x<0$.}
	\label{fig:diff_sigma_f}
\end{figure}

\subsubsection{Different values of $\sigma$ and $\beta$ in the same ratio $\frac{\sigma}{\beta}$}In this section, we vary both coefficients by keeping the ratio $\frac{\sigma}{\beta}$ the same.
We set $\frac{\sigma}{\beta}=\frac{1}{2}$ and vary the coefficients $\sigma$ and $\beta$ as $\sigma=1,0.5,0.25,0.05$ and $\beta=2\sigma$. The convergence to the same physical quantities of the same steady-state in Figure \ref{fig:frac_physical} agrees with the theory, as the steady-state (\ref{maxwellian}) is determined by the fraction of the coefficient $\frac{\sigma}{\beta}$. The figure shows that the coefficient $\sigma=0.05$ is too small to make the quantities converge within 5 time-grids. The difference between the four cases is how fast the kinetic energy and the entropy of DNN solutions converge to those of the steady-state. A larger value of $\beta$ results in the faster convergence though the fractions of the coefficients $\frac{\sigma}{\beta}$ are the same.
\begin{figure}
	\includegraphics[width=0.8\textwidth, draft=false]{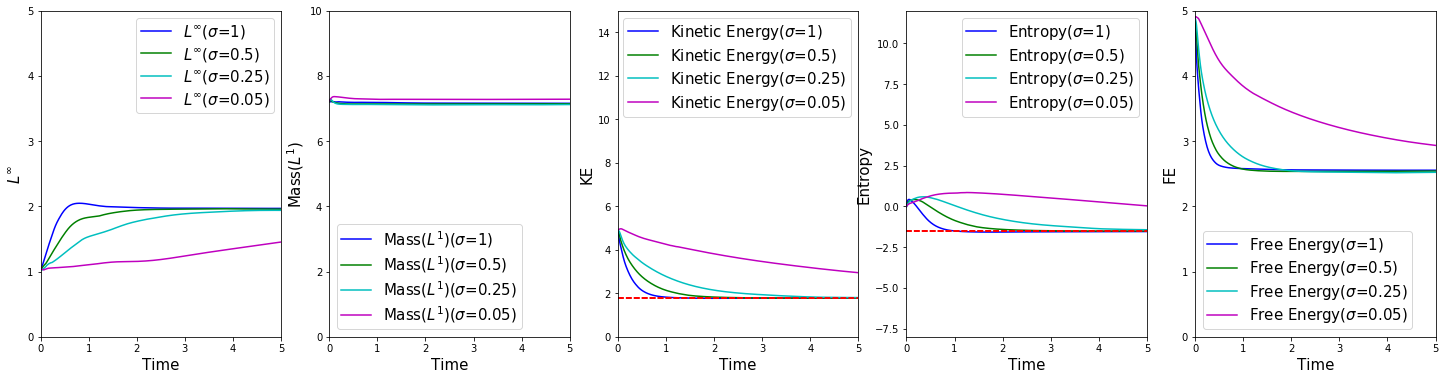}
	\caption{The time-asymptotic behaviors of the $L^\infty$ norms and the macroscopic quantities of $f^{nn}(t,x,v;m,w,b)$ for the specular boundary case as both $\beta$ and $\sigma$ vary by keeping the ratio $\frac{\sigma}{\beta}=\frac{1}{2}.$ The steady-state values of the kinetic energy (\ref{KE_inf}) and the entropy (\ref{Ent_inf}) are also given via the red-dotted lines. It is notable that the free energy (Lyapunov functional) is monotonically decreasing.}
	\label{fig:frac_physical}
\end{figure}
Figure \ref{fig:frac_f} shows the different rate of the convergence to the steady-state; i.e., the four graphs in each plot converge to the same maxwellian but at different rates.
\begin{figure}
	\includegraphics[width=0.8\textwidth, draft=false]{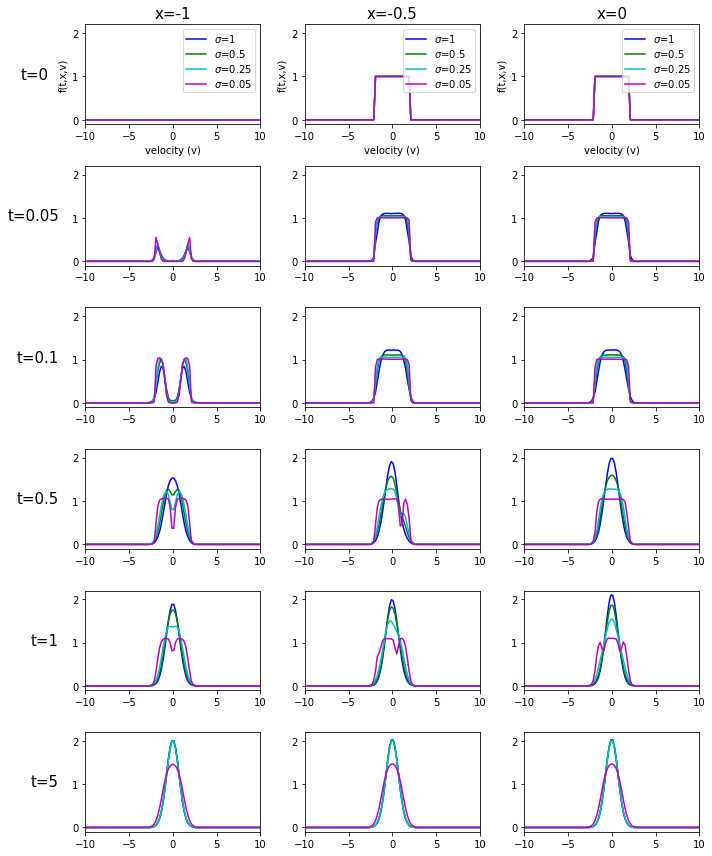}
	\caption{The pointwise values of $f^{nn}(t,x,v;m,w,b)$ as $t$ varies at each $x$'s for the specular boundary case, where both $\sigma$ and $\beta$ vary with the ratio fixed as $\frac{\sigma}{\beta} =\frac{1}{2}$. The plots are shown in different colors as in the legend. $x=-1$ stands for the boundary point, and $x=-0.5$ or $=0$ are the points away from the boundary. We omit the graphs for $x>0$ as we have obtained the symmetric graphs to $x<0$.}
	\label{fig:frac_f}
\end{figure}

\section{Conclusion}\label{sec:conclusion}This paper presents an approximation of the 1D Fokker-Planck equation solution with the inflow-type or the reflection-type boundary conditions. We first provide a proof for the existence of a sequence of DNN solutions such that the total loss, defined in terms of the difference between the $\widehat{C}^{(1,1,2)}$-type solutions and the neural network solutions, converges to zero.  We then provide a proof for the $L^2_{x,v}$ convergence of the sequence of neural network solutions to the actual $\widehat{C}^{(1,1,2)}$-type solutions to the original Fokker-Planck equation in a bounded interval, in the case that the total sum of the loss terms of our own goes to zero. 
In other words, it has been shown that we can reduce the predefined loss function via the appropriate selection of DNN model parameters. Also, via choosing the model parameters that reduce the loss function as desired, it has been shown that the neural network solutions converge to the analytic solution by obtaining an $L^2$ energy estimate based on the adaptation of the transformation of the function in the sense of \cite{carrillo1998global}.

For the neural network simulations, our DNN algorithm uses the library \textit{PyTorch} and the hyper-tangent \textit{tanh} activation function between each layer. Regarding the optimization algorithm, we use \textit{Adam} optimizer which is an extended algorithm of the stochastic gradient descent in the process of minimizing the actual loss in the case of the reflection-type boundary conditions. The reduction of the order and the weight sharing between $f^{nn}$ and $\partial_v f^{nn}$ dramatically reduced the learning cost. We remark that we also increased the learning efficiency of the model by adding the mass conservation law to the total loss function. In addition, we have numerically confirmed several theoretical predictions on the asymptotic behaviors of the solutions to the equation and the entropy production; i.e., we were able to provide the numerical simulations on the pointwise convergence of the neural network solutions to the global Maxwellian for varied types of the boundary conditions and the varied coefficients for the diffusion and the friction and observed that the free energy (the relative entropy) is monotonically decreasing. 

The possible extensions of our work for the future include the generalizations to other various types of the kinetic equations. These include the multi-dimensional (fractional) Fokker-Planck equation, the coupled Vlasov equation such as the Vlasov-Poisson or the Vlasov-Maxwell systems, and the Boltzmann-type collisional kinetic equations. We expect that our work can further be extended to the generalized situations listed above via the adaptation of the sampling of random grids to our learning algorithm so the total cost for the learning process of the DNN algorithm is being effectively reduced while the total loss function is being properly reduced at the same time. 

\section*{Acknowledgement}
H. J. Hwang, H. Jo, and J. Y. Lee were supported by the Basic Science
Research Program through the National Research Foundation of Korea (NRF- 2017R1E1A1A\\ 03070105 and NRF-2019R1A5A1028324). J. W. Jang was supported by the Korean IBS project IBS-R003-D1. In addition, J. W. Jang gratefully acknowledge the support of the Hausdorff Research Institute for Mathematics (Bonn), through the Junior Trimester Program on Kinetic Theory.



\bibliographystyle{amsplaindoi} 
\bibliography{bibliography}

\providecommand{\bysame}{\leavevmode\hbox to3em{\hrulefill}\thinspace}
\providecommand{\MR}{\relax\ifhmode\unskip\space\fi MR }
\providecommand{\MRhref}[2]{%
  \href{http://www.ams.org/mathscinet-getitem?mr=#1}{#2}
}
\providecommand{\href}[2]{#2}
\begin{thebibliography}{10}

\bibitem{allen1994computational}
EJ~Allen and HD~Victory~Jr, \emph{A computational investigation of the random
  particle method for numerical solution of the kinetic
  vlasov-poisson-fokker-planck equations}, Physica A: Statistical Mechanics and
  its Applications \textbf{209} (1994), no.~3-4, 318--346.

\bibitem{MR1848592}
A.~Arnold, J.~A. Carrillo, I.~Gamba, and C.-W. Shu,
  \href{https://doi.org/10.1081/TT-100105365}{\emph{Low and high field scaling
  limits for the {V}lasov- and {W}igner-{P}oisson-{F}okker-{P}lanck systems}},
  vol.~30, 2001, The Sixteenth International Conference on Transport Theory,
  Part I (Atlanta, GA, 1999), pp.~121--153. \MR{1848592}

\bibitem{MR2974172}
Anton Arnold, Irene~M. Gamba, Maria~Pia Gualdani, St\'{e}phane Mischler,
  Clement Mouhot, and Christof Sparber,
  \href{https://doi.org/10.1142/S0218202512500340}{\emph{The
  {W}igner-{F}okker-{P}lanck equation: stationary states and large time
  behavior}}, Math. Models Methods Appl. Sci. \textbf{22} (2012), no.~11,
  1250034, 31. \MR{2974172}

\bibitem{MR892257}
Yu.~A. Berezin, V.~N. Khudick, and M.~S. Pekker,
  \href{https://doi.org/10.1016/0021-9991(87)90160-4}{\emph{Conservative
  finite-difference schemes for the {F}okker-{P}lanck equation not violating
  the law of an increasing entropy}}, J. Comput. Phys. \textbf{69} (1987),
  no.~1, 163--174. \MR{892257}

\bibitem{berg2018unified}
Jens Berg and Kaj Nystr{\"o}m, \emph{A unified deep artificial neural network
  approach to partial differential equations in complex geometries},
  Neurocomputing \textbf{317} (2018), 28--41.

\bibitem{MR610857}
A.~V. Bobyl\"{e}v, I.~F. Potapenko, and V.~A. Chuyanov, \emph{Completely
  conservative difference schemes for nonlinear kinetic equations of {L}andau
  ({F}okker-{P}lanck) type}, Akad. Nauk SSSR Inst. Prikl. Mat. Preprint (1980),
  no.~76, 26. \MR{610857}

\bibitem{MR1470927}
L.~L. Bonilla, J.~A. Carrillo, and J.~Soler,
  \href{https://doi.org/10.1137/S0036139995291544}{\emph{Asymptotic behavior of
  an initial-boundary value problem for the
  {V}lasov-{P}oisson-{F}okker-{P}lanck system}}, SIAM J. Appl. Math.
  \textbf{57} (1997), no.~5, 1343--1372. \MR{1470927}

\bibitem{bouchut1993existence}
Fran{\c{c}}ois Bouchut, \emph{Existence and uniqueness of a global smooth
  solution for the vlasov-poisson-fokker-planck system in three dimensions},
  Journal of functional analysis \textbf{111} (1993), no.~1, 239--258.

\bibitem{bouchut1995smoothing}
\bysame, \emph{Smoothing effect for the non-linear vlasov-poisson-fokker-planck
  system}, Journal of differential equations \textbf{122} (1995), no.~2,
  225--238.

\bibitem{bouchut1995long}
Fran{\c{c}}ois Bouchut, Jean Dolbeault, et~al., \emph{On long time asymptotics
  of the vlasov-fokker-planck equation and of the vlasov-poisson-fokker-planck
  system with coulombic and newtonian potentials}, Differential and Integral
  Equations \textbf{8} (1995), no.~3, 487--514.

\bibitem{bris2008existence}
C~Le Bris and P-L Lions, \emph{Existence and uniqueness of solutions to
  fokker--planck type equations with irregular coefficients}, Communications in
  Partial Differential Equations \textbf{33} (2008), no.~7, 1272--1317.

\bibitem{MR1640174}
C.~Buet and S.~Cordier,
  \href{https://doi.org/10.1006/jcph.1998.6015}{\emph{Conservative and entropy
  decaying numerical scheme for the isotropic {F}okker-{P}lanck-{L}andau
  equation}}, J. Comput. Phys. \textbf{145} (1998), no.~1, 228--245.
  \MR{1640174}

\bibitem{MR1688993}
\bysame, \href{https://doi.org/10.1137/S0036142997322102}{\emph{Numerical
  analysis of conservative and entropy schemes for the
  {F}okker-{P}lanck-{L}andau equation}}, SIAM J. Numer. Anal. \textbf{36}
  (1999), no.~3, 953--973. \MR{1688993}

\bibitem{MR1447091}
C.~Buet, S.~Cordier, P.~Degond, and M.~Lemou,
  \href{https://doi.org/10.1006/jcph.1997.5669}{\emph{Fast algorithms for
  numerical, conservative, and entropy approximations of the
  {F}okker-{P}lanck-{L}andau equation}}, J. Comput. Phys. \textbf{133} (1997),
  no.~2, 310--322. \MR{1447091}

\bibitem{MR1639292}
J.~A. Carrillo and G.~Toscani,
  \href{https://doi.org/10.1002/(SICI)1099-1476(19980910)21:13<1269::AID-MMA995>3.3.CO;2-F}{\emph{Exponential
  convergence toward equilibrium for homogeneous {F}okker-{P}lanck-type
  equations}}, Math. Methods Appl. Sci. \textbf{21} (1998), no.~13, 1269--1286.
  \MR{1639292}

\bibitem{carrillo1998global}
Jos{\'e}~A Carrillo, \emph{Global weak solutions for the
  initial--boundary-value problems vlasov--poisson--fokker--planck system},
  Mathematical methods in the applied sciences \textbf{21} (1998), no.~10,
  907--938.

\bibitem{MR2765745}
Jos\'{e}~A. Carrillo, Renjun Duan, and Ayman Moussa,
  \href{https://doi.org/10.3934/krm.2011.4.227}{\emph{Global classical
  solutions close to equilibrium to the {V}lasov-{F}okker-{P}lanck-{E}uler
  system}}, Kinet. Relat. Models \textbf{4} (2011), no.~1, 227--258.
  \MR{2765745}

\bibitem{MR1343393}
Jos\'{e}~A. Carrillo and Juan Soler,
  \href{https://doi.org/10.1002/mma.1670181006}{\emph{On the initial value
  problem for the {V}lasov-{P}oisson-{F}okker-{P}lanck system with initial data
  in {$L^p$} spaces}}, Math. Methods Appl. Sci. \textbf{18} (1995), no.~10,
  825--839. \MR{1343393}

\bibitem{MR1414375}
Jos\'{e}~A. Carrillo, Juan Soler, and Juan~Luis V\'{a}zquez,
  \href{https://doi.org/10.1006/jfan.1996.0123}{\emph{Asymptotic behaviour and
  self-similarity for the three-dimensional
  {V}lasov-{P}oisson-{F}okker-{P}lanck system}}, J. Funct. Anal. \textbf{141}
  (1996), no.~1, 99--132. \MR{1414375}

\bibitem{MR1739113}
L.~Chac\'{o}n, D.~C. Barnes, D.~A. Knoll, and G.~H. Miley,
  \href{https://doi.org/10.1006/jcph.1999.6394}{\emph{An implicit
  energy-conservative 2{D} {F}okker-{P}lanck algorithm. {I}. {D}ifference
  scheme}}, J. Comput. Phys. \textbf{157} (2000), no.~2, 618--653. \MR{1739113}

\bibitem{cheng1976integration}
Chio-Zong Cheng and Georg Knorr, \emph{The integration of the vlasov equation
  in configuration space}, Journal of Computational Physics \textbf{22} (1976),
  no.~3, 330--351.

\bibitem{cotter1990stone}
Neil~E Cotter, \emph{The stone-weierstrass theorem and its application to
  neural networks}, IEEE Transactions on Neural Networks \textbf{1} (1990),
  no.~4, 290--295.

\bibitem{MR2186367}
N.~Crouseilles and F.~Filbet, \href{https://doi.org/10.4171/012-1/4}{\emph{A
  conservative and entropic method for the {V}lasov-{F}okker-{P}lanck-{L}andau
  equation}}, Numerical methods for hyperbolic and kinetic problems, IRMA Lect.
  Math. Theor. Phys., vol.~7, Eur. Math. Soc., Z\"{u}rich, 2005, pp.~59--70.
  \MR{2186367}

\bibitem{cybenko1989approximation}
George Cybenko, \emph{Approximation by superpositions of a sigmoidal function},
  Mathematics of control, signals and systems \textbf{2} (1989), no.~4,
  303--314.

\bibitem{degond1986global}
Pierre Degond, \emph{Global existence of smooth solutions for the
  vlasov-fokker-planck equation in $1 $ and $2 $ space dimensions}, Annales
  scientifiques de l'{\'E}cole Normale Sup{\'e}rieure, vol.~19, 1986,
  pp.~519--542.

\bibitem{MR1283340}
Pierre Degond and Brigitte Lucquin-Desreux,
  \href{https://doi.org/10.1007/s002110050059}{\emph{An entropy scheme for the
  {F}okker-{P}lanck collision operator of plasma kinetic theory}}, Numer. Math.
  \textbf{68} (1994), no.~2, 239--262. \MR{1283340}

\bibitem{desvillettes2001trend}
Laurent Desvillettes and C{\'e}dric Villani, \emph{On the trend to global
  equilibrium in spatially inhomogeneous entropy-dissipating systems: The
  linear fokker-planck equation}, Communications on Pure and Applied
  Mathematics: A Journal Issued by the Courant Institute of Mathematical
  Sciences \textbf{54} (2001), no.~1, 1--42.

\bibitem{diperna1988fokker}
RJ~DiPerna and PL~Lions, \emph{On the fokker-planck-boltzmann equation},
  Communications in mathematical physics \textbf{120} (1988), no.~1, 1--23.

\bibitem{Dita_1985}
P~Dita, \href{http://dx.doi.org/10.1088/0305-4470/18/14/019}{\emph{The
  fokker-planck equation with absorbing boundary}}, Journal of Physics A:
  Mathematical and General \textbf{18} (1985), no.~14, 2685--2690.

\bibitem{MR1910805}
F.~Filbet and L.~Pareschi,
  \href{https://doi.org/10.1007/978-3-662-04784-2_45}{\emph{Numerical solution
  of the non homogeneous {F}okker-{P}lanck-{L}andau equation}}, Progress in
  industrial mathematics at {ECMI} 2000 ({P}alermo), Math. Ind., vol.~1,
  Springer, Berlin, 2002, pp.~325--331. \MR{1910805}

\bibitem{filbet2002numerical}
Francis Filbet and Lorenzo Pareschi, \emph{A numerical method for the accurate
  solution of the fokker--planck--landau equation in the nonhomogeneous case},
  Journal of Computational Physics \textbf{179} (2002), no.~1, 1--26.

\bibitem{MR1906573}
\bysame, \href{https://doi.org/10.1006/jcph.2002.7010}{\emph{A numerical method
  for the accurate solution of the {F}okker-{P}lanck-{L}andau equation in the
  nonhomogeneous case}}, J. Comput. Phys. \textbf{179} (2002), no.~1, 1--26.
  \MR{1906573}

\bibitem{MR3519972}
Irene~M. Gamba and Moon-Jin Kang,
  \href{https://doi.org/10.1007/s00205-016-1002-2}{\emph{Global weak solutions
  for {K}olmogorov-{V}icsek type equations with orientational interactions}},
  Arch. Ration. Mech. Anal. \textbf{222} (2016), no.~1, 317--342. \MR{3519972}

\bibitem{han2018solving}
Jiequn Han, Arnulf Jentzen, and E~Weinan, \emph{Solving high-dimensional
  partial differential equations using deep learning}, Proceedings of the
  National Academy of Sciences \textbf{115} (2018), no.~34, 8505--8510.

\bibitem{havlak1996numerical}
Karl~J Havlak and Harold~Dean Victory, Jr, \emph{The numerical analysis of
  random particle methods applied to vlasov--poisson fokker-planck kinetic
  equations}, SIAM journal on numerical analysis \textbf{33} (1996), no.~1,
  291--317.

\bibitem{havlak1998deterministic}
Karl~J Havlak and Harold~Dean Victory~Jr, \emph{On deterministic particle
  methods for solving vlasov--poisson--fokker--planck systems}, SIAM journal on
  numerical analysis \textbf{35} (1998), no.~4, 1473--1519.

\bibitem{hornik1989multilayer}
Kurt Hornik, Maxwell Stinchcombe, and Halbert White, \emph{Multilayer
  feedforward networks are universal approximators}, Neural networks \textbf{2}
  (1989), no.~5, 359--366.

\bibitem{hwang2018fokker}
Hyung~Ju Hwang, Juhi Jang, and Jaewoo Jung, \emph{The fokker--planck equation
  with absorbing boundary conditions in bounded domains}, SIAM Journal on
  Mathematical Analysis \textbf{50} (2018), no.~2, 2194--2232.

\bibitem{hwang2014fokker}
Hyung~Ju Hwang, Juhi Jang, and Juan~JL Vel{\'a}zquez, \emph{The fokker--planck
  equation with absorbing boundary conditions}, Archive for Rational Mechanics
  and Analysis \textbf{214} (2014), no.~1, 183--233.

\bibitem{hwang2017fokker}
Hyung~Ju Hwang and Du~Phan, \emph{On the fokker--planck equations with inflow
  boundary conditions}, Quarterly of Applied Mathematics \textbf{75} (2017),
  no.~2, 287--308.

\bibitem{jianyu2003numerical}
Li~Jianyu, Luo Siwei, Qi~Yingjian, and Huang Yaping, \emph{Numerical solution
  of elliptic partial differential equation using radial basis function neural
  networks}, Neural Networks \textbf{16} (2003), no.~5-6, 729--734.

\bibitem{MR3780745}
Shi Jin and Yuhua Zhu,
  \href{https://doi.org/10.1137/17M1123845}{\emph{Hypocoercivity and uniform
  regularity for the {V}lasov-{P}oisson-{F}okker-{P}lanck system with
  uncertainty and multiple scales}}, SIAM J. Math. Anal. \textbf{50} (2018),
  no.~2, 1790--1816. \MR{3780745}

\bibitem{jo2019deep}
Hyeontae Jo, Hwijae Son, Hyung~Ju Hwang, and Eunheui Kim, \emph{Deep neural
  network approach to forward-inverse problems}, arXiv preprint
  arXiv:1907.12925 (2019).

\bibitem{lagaris1998artificial}
Isaac~E Lagaris, Aristidis Likas, and Dimitrios~I Fotiadis, \emph{Artificial
  neural networks for solving ordinary and partial differential equations},
  IEEE transactions on neural networks \textbf{9} (1998), no.~5, 987--1000.

\bibitem{lagaris2000neural}
Isaac~E Lagaris, Aristidis~C Likas, and Dimitris~G Papageorgiou,
  \emph{Neural-network methods for boundary value problems with irregular
  boundaries}, IEEE Transactions on Neural Networks \textbf{11} (2000), no.~5,
  1041--1049.

\bibitem{MR816660}
E.~W. Larsen, C.~D. Levermore, G.~C. Pomraning, and J.~G. Sanderson,
  \href{https://doi.org/10.1016/0021-9991(85)90070-1}{\emph{Discretization
  methods for one-dimensional {F}okker-{P}lanck operators}}, J. Comput. Phys.
  \textbf{61} (1985), no.~3, 359--390. \MR{816660}

\bibitem{MR1606249}
M.~Lemou, \href{https://doi.org/10.1007/s002110050327}{\emph{Multipole
  expansions for the {F}okker-{P}lanck-{L}andau operator}}, Numer. Math.
  \textbf{78} (1998), no.~4, 597--618. \MR{1606249}

\bibitem{li1996simultaneous}
Xin Li, \emph{Simultaneous approximations of multivariate functions and their
  derivatives by neural networks with one hidden layer}, Neurocomputing
  \textbf{12} (1996), no.~4, 327--343.

\bibitem{lorenz2007radon}
Thomas Lorenz, \emph{Radon measures solving the cauchy problem of the nonlinear
  transport equation},  (2007).

\bibitem{mcculloch1943logical}
Warren~S McCulloch and Walter Pitts, \emph{A logical calculus of the ideas
  immanent in nervous activity}, The bulletin of mathematical biophysics
  \textbf{5} (1943), no.~4, 115--133.

\bibitem{mischler2010kinetic}
St{\'e}phane Mischler, \emph{Kinetic equations with maxwell boundary
  conditions}, Annales Scientifiques de l'Ecole Normale Superieure, vol.~43,
  2010, pp.~719--760.

\bibitem{neunzert1984vlasov}
Helmut Neunzert, Mario Pulvirenti, and Livio Triolo, \emph{On the
  vlasov-fokker-planck equation}, Mathematical methods in the applied sciences
  \textbf{6} (1984), no.~1, 527--538.

\bibitem{pareschi2000fast}
Lorenzo Pareschi, G~Russo, and G~Toscani, \emph{Fast spectral methods for the
  fokker--planck--landau collision operator}, Journal of Computational Physics
  \textbf{165} (2000), no.~1, 216--236.

\bibitem{MR1677589}
I.~F. Potapenko and C.~A. de~Azevedo,
  \href{https://doi.org/10.1016/S0377-0427(98)00245-3}{\emph{The completely
  conservative difference schemes for the nonlinear {L}andau-{F}okker-{P}lanck
  equation}}, vol. 103, 1999, Applied and computational topics in partial
  differential equations (Gramado, 1997), pp.~115--123. \MR{1677589}

\bibitem{Protopopescu_1987}
V~Protopopescu, \href{http://dx.doi.org/10.1088/0305-4470/20/18/004}{\emph{On
  the fokker-planck equation with force term}}, Journal of Physics A:
  Mathematical and General \textbf{20} (1987), no.~18, L1239--L1244.

\bibitem{raissi2019physics}
Maziar Raissi, Paris Perdikaris, and George~E Karniadakis,
  \emph{Physics-informed neural networks: A deep learning framework for solving
  forward and inverse problems involving nonlinear partial differential
  equations}, Journal of Computational Physics \textbf{378} (2019), 686--707.

\bibitem{raissi2017physics}
Maziar Raissi, Paris Perdikaris, and George~Em Karniadakis, \emph{Physics
  informed deep learning (part i): Data-driven solutions of nonlinear partial
  differential equations}, arXiv preprint arXiv:1711.10561 (2017).

\bibitem{schaeffer1997difference}
Jack Schaeffer, \emph{A difference scheme for the vlasov-poisson-fokker-planck
  system},  (1997).

\bibitem{schaeffer1998convergence}
\bysame, \emph{Convergence of a difference scheme for the
  vlasov--poisson--fokker--planck system in one dimension}, SIAM journal on
  numerical analysis \textbf{35} (1998), no.~3, 1149--1175.

\bibitem{sheng2013well}
Qiwei Sheng and Weimin Han, \emph{Well-posedness of the fokker-planck equation
  in a scattering process}, Journal of Mathematical Analysis and Applications
  \textbf{406} (2013), no.~2, 531--536.

\bibitem{siahkoohi2019neural}
Ali Siahkoohi, Mathias Louboutin, and Felix~J Herrmann, \emph{Neural network
  augmented wave-equation simulation}, arXiv preprint arXiv:1910.00925 (2019).

\bibitem{sirignano2018dgm}
Justin Sirignano and Konstantinos Spiliopoulos, \emph{Dgm: A deep learning
  algorithm for solving partial differential equations}, Journal of
  Computational Physics \textbf{375} (2018), 1339--1364.

\bibitem{victory1990classical}
Harold~Dean Victory~Jr and Brian~P O'Dwyer, \emph{On classical solutions of
  vlasov-poisson fokker-planck systems}, Indiana University mathematics journal
  (1990), 105--156.

\bibitem{wei2019general}
Shiyin Wei, Xiaowei Jin, and Hui Li, \emph{General solutions for nonlinear
  differential equations: a rule-based self-learning approach using deep
  reinforcement learning}, Computational Mechanics (2019), 1--14.

\bibitem{wollman2005numerical}
Stephen Wollman and Ercument Ozizmir, \emph{Numerical approximation of the
  vlasov--poisson--fokker--planck system in one dimension}, Journal of
  Computational Physics \textbf{202} (2005), no.~2, 602--644.

\bibitem{wollman2008deterministic}
\bysame, \emph{A deterministic particle method for the vlasov--fokker--planck
  equation in one dimension}, Journal of Computational and Applied Mathematics
  \textbf{213} (2008), no.~2, 316--365.

\bibitem{wollman2009numerical}
\bysame, \emph{Numerical approximation of the vlasov--poisson--fokker--planck
  system in two dimensions}, Journal of Computational Physics \textbf{228}
  (2009), no.~18, 6629--6669.

\bibitem{MR3715369}
Yuhua Zhu and Shi Jin, \href{https://doi.org/10.1137/16M1090028}{\emph{The
  {V}lasov-{P}oisson-{F}okker-{P}lanck system with uncertainty and a
  one-dimensional asymptotic preserving method}}, Multiscale Model. Simul.
  \textbf{15} (2017), no.~4, 1502--1529. \MR{3715369}

\end{thebibliography}





\end{document}